\documentclass[12pt]{amsart}
\usepackage{amssymb, hyperref, mathrsfs,xcolor}

\usepackage[margin=1.0in]{geometry}

\newtheorem{theoremIntro}{Theorem}[]

\newtheorem{questionIntro}{Question}

\newtheorem{theorem}{Theorem}[section]
\newtheorem{lemma}[theorem]{Lemma}
\newtheorem{proposition}[theorem]{Proposition}

\newtheorem{definition}[theorem]{Definition}
\newtheorem{corollary}[theorem]{Corollary}

\theoremstyle{remark}
\newtheorem{remark}[theorem]{Remark}
\newtheorem{example}[theorem]{Example}

\numberwithin{equation}{section}

\newcommand{\calU}{\ensuremath{\mathcal{U}}}
\newcommand{\calQ}{\ensuremath{\mathcal{Q}}}

\newcommand{\calS}{\ensuremath{\mathcal{S}}}
\newcommand{\calH}{\ensuremath{\mathcal{H}}}

\newcommand{\fin}{\ensuremath{\mathrm{fin}}}

\newcommand{\tors}{\ensuremath{\mathrm{tors}}}

\newcommand{\calC}{\ensuremath{\mathcal{C}}}
\newcommand{\calD}{\ensuremath{\mathcal{D}}}
\newcommand{\calM}{\ensuremath{\mathcal{M}}}
\newcommand{\calN}{\ensuremath{\mathcal{N}}}
\newcommand{\calW}{\ensuremath{\mathcal {W}}}
\newcommand{\calL}{\ensuremath{\mathcal {L}}}

\newcommand{\frakL}{\ensuremath{\mathfrak{L}}}
\newcommand{\frakl}{\ensuremath{\mathfrak{l}}}

\newcommand{\hra}{\hookrightarrow}
\newcommand{\xra}{\xrightarrow}

\newcommand{\modu}{\ensuremath{\mathrm{\;mod\;}}}
\newcommand{\mo}{{-1}}
\newcommand{\scrC}{\ensuremath{\mathscr{C}}}
\newcommand{\scrQ}{\ensuremath{\mathscr{Q}}}
\newcommand{\scrW}{\ensuremath{\mathscr{W}}}

\newcommand{\bbZ}{\ensuremath{\mathbb{Z}}}

\newcommand{\bbR}{\ensuremath{\mathbb{R}}}
\newcommand{\bbN}{\ensuremath{\mathbb{N}}}

\begin{document}

\title{On minimal complements in groups}

\author{Arindam Biswas}
\address{Universit\"at Wien, Fakult\"at f\"ur Mathematik, 
Oskar-Morgenstern-Platz 1, 1090 Wien, Austria.}
\curraddr{Department of Mathematics, Technion - Israel Institute of Technology, Haifa 32000, Israel}
\email{biswas@campus.technion.ac.il}
\thanks{}

\author{Jyoti Prakash Saha}
\address{Department of Mathematics, Indian Institute of Science Education and Research Bhopal, Bhopal Bypass Road, Bhauri, Bhopal 462066, Madhya Pradesh,
India}
\curraddr{}
\email{jpsaha@iiserb.ac.in}
\thanks{}



\begin{abstract}
Let $W,W'\subseteq G$ be nonempty subsets in an arbitrary group $G$. The set $W'$ is said to be a complement to $W$ if $WW'=G$ and it is minimal if no proper subset of $W'$ is a complement to $W$. We show that, if $W$ is finite then every complement of $W$ has a minimal complement, answering a problem of Nathanson. This also shows the existence of minimal $r$-nets for every $r\geqslant 0$ in finitely generated groups. Further, we give necessary and sufficient conditions for the existence of minimal complements of a certain class of infinite subsets in finitely generated abelian groups, partially answering another problem of Nathanson. Finally, we provide infinitely many examples of infinite subsets of abelian groups of arbitrary finite rank admitting minimal complements.

\end{abstract}

\maketitle

\tableofcontents

\vspace*{-10mm}
\section{Introduction}

\subsection{Motivation}

Let $(G,.)$ be a group and $W\subseteq G$ be a nonempty subset. A nonempty set $W'\subseteq G$ is said to be a complement to $W$ if $$W\cdot W' = G.$$
Let $\mathcal{W}$ denote the set of all complements of $W$. Then it is clear that $\mathcal{W}\neq \emptyset$ (since $G\in \mathcal{W}$) and also the fact that the elements of $\mathcal{W}$ form a partially ordered set under inclusion.

\begin{definition}[minimal complement]
A complement $W'$ to $W$ is minimal if no proper subset of $W'$ is a complement to $W$, i.e., 
$$W\cdot W' = G \,\text{ and }\, W\cdot (W'\setminus \lbrace w'\rbrace)\subsetneq G \,\,\, \forall w'\in W'.$$
\end{definition}
 From now on, whenever we shall use the word ``complement" we shall mean additive or multiplicative complement (depending on the operation in the ambient group) and not the set theoretic complement unless explicitly mentioned otherwise. We will omit the symbol ``$\,\cdot\,$'' and denote $A\cdot B$ by $AB$. If $A$ (resp. $B$) contains only one element $a$ (resp. $b$), then $AB$ is denoted by $aB$ (resp. $Ab$). Moreover, in case of abelian groups we shall sometimes use $+$ to denote the group operation. 
 
 Given a minimal complement $W'$ of $W$, we see that the right translation $W'g$ is also a minimal complement of $W$ and $W'$ is a minimal complement of $gW$ for all $g \in G$. Thus, the existence of a minimal complement of a nonempty subset is equivalent to the existence of a minimal complement of any of its left translates.

It was shown by Nathanson (see \cite[Theorem 8]{NathansonAddNT4}) that for a non-empty, finite subset $W$ in the additive group $\mathbb{Z}$, any complement to $W$ contains a minimal complement. In the same paper he asked the following questions:

\begin{questionIntro}
\cite[Problem 11]{NathansonAddNT4}
\label{nathansonprob11}
``Let $W$ be an infinite set of integers. Does there exist a minimal
complement to $W$? Does there exist a complement to $W$ that does not contain a
minimal complement?''
\end{questionIntro}

\begin{questionIntro}
\cite[Problem 12]{NathansonAddNT4}
\label{nathansonprob12}
``Let $G$ be an infinite group, and let $W$ be a finite subset of $G$. Does
there exist a minimal complement to $W$? Does there exist a complement to $W$ that
does not contain a minimal complement?''
\end{questionIntro}

\begin{questionIntro}
\cite[Problem 13]{NathansonAddNT4} \label{nathansonprob13}
``Let $G$ be an infinite group, and let $W$ be an infinite subset of $G$.
Does there exist a minimal complement to $W$? Does there exist a complement to
$W$ that does not contain a minimal complement?''
\end{questionIntro}

Since then the problems have generated considerable interest. Chen and Yang in 2012 gave examples of two infinite sets $W_1,W_2 \subset \mathbb{Z}$, such that $W_1$ has a complement that does not contain a minimal complement and every complement to $W_2$ contains a minimal complement (see \cite[Remark 1]{ChenYang12}). They also gave certain necessary and certain sufficient conditions on the infinite set $W \subset \mathbb{Z}$ such that $W$ has a minimal complement (see \cite[Theorems 1, 2]{ChenYang12}). Very recently, Kiss, S\'{a}ndor and Yang \cite{KissSandorYangJCT19} succeeded in giving necessary and sufficient conditions for the existence of minimal complements of several other class of infinite sets in $\mathbb{Z}$ (which were not covered in the previous work of Chen and Yang). See \cite[Theorems 1, 2, 3]{KissSandorYangJCT19}. In \cite{RuzsaLacunary}, Ruzsa studied additive complements (defined in a different sense) of lacunary sequences. 

\subsection{Statement of results}

All the aforementioned progresses were in the setting of Question \ref{nathansonprob11}. In this article, we deal with the Questions \ref{nathansonprob12} and \ref{nathansonprob13}. Specifically, we show the following results: 
\begin{theoremIntro}
\label{theorem1.1}
Let $G$ be an arbitrary group with $S$ a nonempty finite subset of $G$. Then every complement of $S$ in $G$ has a minimal complement.
\end{theoremIntro}
This answers Question \ref{nathansonprob12} of Nathanson. Next, we turn to Question \ref{nathansonprob13}. Before commencing the discussion in detail, we state that it has a simple answer in certain cases, for instance, when $W$ is a subgroup or when $W$ behaves well under the quotient map by a subgroup. Namely, in these cases a minimal complement always exist. See Proposition \ref{prop4.1} and Lemma \ref{Lemma: coset one}. For general infinite subsets, the situation is more delicate. To give an answer to Question \ref{nathansonprob13}, one needs to consider the infinite subsets which have less algebraic structure. 

Our goal in this case is to establish the existence of minimal complements for a large class of infinite sets in finitely generated abelian groups. This will give the claimed partial solution to Question \ref{nathansonprob13} of Nathanson. In section \ref{Sec: Minimal complement}, we focus on the minimal complements of certain infinite subsets of free abelian groups of finite rank, which are of the form $\bbZ^d$ for some integer $d\geqslant 1$. It is interesting to consider the subsets $X$ of $\bbZ^d$ such that $x + \bbN u_1 + \cdots + \bbN u_d$ is contained in $X$ for any $x\in X$, i.e.,
\begin{equation}
\label{Eqn: preperiodic}
X \supseteq X + \bbN u_1 + \cdots + \bbN u_d,
\end{equation}
where $u_1, \cdots, u_d$ are elements of $\bbZ^d$ satisfying no nontrivial $\bbZ$-linear relation. However it turns out that such sets do not necessarily have a minimal complement. For instance, $\bbZ^{d-1} \times \bbN$ is a subset of $\bbZ^d$ which satisfies Equation \eqref{Eqn: preperiodic} with $u_i = e_i$, but do not have any minimal complement. 

To obtain examples of infinite subsets of $\bbZ^d$ having minimal complements, we consider the periodic subsets of $\bbZ^d$ (these are subsets of $\bbZ^d$ satisfying Equation \eqref{Eqn: preperiodic} along with a finiteness condition, see Definition \ref{Defn: periodic}). Unfortunately, there exist periodic subsets of $\bbZ^d$, which do not admit any minimal complement (see Proposition \ref{Cor: Nd minimal complement}). So we consider a more general class of subsets of $\bbZ^d$ satisfying a weaker version of Equation \eqref{Eqn: preperiodic} along with certain finiteness condition (which we call \textit{eventually periodic subsets},
see Definition \ref{Defn: periodic} - these are $d$-dimensional analogues of the eventually periodic sets in $\mathbb{Z}$ considered by Kiss--S\'andor--Yang in \cite{KissSandorYangJCT19}). Given an eventually periodic subset $W$ of $\bbZ^d$ (with periods $u_1, \cdots, u_d$), by Theorem \ref{Thm: Structure of Eventually Periodic}, there exist subsets $\calW, \scrW$ of $W$ such that the equality 
\begin{equation}
W  = \scrW \sqcup (\calW + (\bbN u_1 + \cdots + \bbN u_d))
\end{equation}
holds. It turns out that several (in fact, infinitely many) eventually periodic subsets of $\bbZ^d$ have minimal complements (see Proposition \ref{Prop:MinComp of even perio mod 2} and Example \ref{Example infinite}). The following result provides a necessary condition for an eventually periodic subset of $\bbZ^d$ to have a minimal complement. 

\begin{theoremIntro}
[Theorem \ref{Thm: existence of min complement implies}]
\label{theorem1.1.2}
Let $W$ be an eventually periodic subset of $\bbZ^d$ with periods $u_1, \cdots, u_d$. Let $\scrW_1$ be as in Theorem \ref{Thm: existence of min complement implies}. Suppose $W$ has a minimal complement in $\bbZ^d$. Then $\scrW_1$ is nonempty and there exists a nonempty finite subset $\calM$ of $\bbZ^d$ such that the following conditions hold. 
\begin{enumerate}
\item The quotient map $\pi: \calM \to \bbZ^d/\calL$ is injective.
\item The quotient map $\pi: (\calM + (\calW \cup \scrW_1)) \to \bbZ^d/\calL$ is surjective.
\item For any $m\in \calM$, there exists $w\in \scrW_1$ such that $m + w \not\equiv m' + w' \modu \calL$ for any $m'\in \calM, w'\in \calW$.
\end{enumerate}
\end{theoremIntro}

This is a simplified version of Theorem \ref{Thm: existence of min complement implies}. We refer to Theorem \ref{Thm: existence of min complement implies} for a detailed statement.

As a consequence, we obtain that no periodic subset of $\bbZ^d$ has a minimal complement (Corollary \ref{Cor: No periodic has minimal}). We also prove the result below, giving a sufficient condition for an eventually periodic subset of $\bbZ^d$ to have a minimal complement. 

\begin{theoremIntro}
[Theorem \ref{Thm: Implies existence of min comple}]
\label{theorem1.1.3}
Let $W$ be an eventually periodic subset of $\bbZ^d$ with periods $u_1, \cdots, u_d$. Let $\scrW_1$ be as in Theorem \ref{Thm: existence of min complement implies}. Suppose $\scrW_1$ is nonempty and there exists a nonempty finite subset $\calM$ of $\bbZ^d$ such that the following conditions hold. 
\begin{enumerate}
\item The quotient map $\pi: \calM\to \bbZ^d/\calL$ is injective.
\item The quotient map $\pi: (\calM + (\calW \cup \scrW_1)) \to \bbZ^d/\calL$ is surjective.
\item For any $m\in \calM$, there exists $w\in \scrW_1$ such that $m + w \not\equiv m' + w' \modu \calL$ for any $m'\in \calM\setminus\{m\}, w'\in \calW\cup \scrW_1$.
\end{enumerate}
Then $W$ has a minimal complement in $\bbZ^d$.
\end{theoremIntro}

Using the above, we get a necessary and sufficient condition for certain types of eventually periodic subsets of $\bbZ^d$ to have a minimal complement.

\begin{theoremIntro}
[Theorem \ref{Thm: Even peri 1 has min comp iff}]
\label{theorem1.1.4}
Let $W$ be an eventually periodic subset of $\bbZ^d$ with periods $u_1, \cdots, u_d$. Let $\scrW_1$ be as in Theorem \ref{Thm: existence of min complement implies}. Suppose $\scrW_1$ contains only one element. Then $W$ has a minimal complement in $\bbZ^d$ if and only if there exists a nonempty finite subset $\calM$ of $\bbZ^d$ satisfying the following. 
\begin{enumerate}
\item The quotient map $\pi: \calM\to \bbZ^d/\calL$ is injective.
\item The quotient map $\pi: (\calM+ (\calW\cup \scrW_1)) \to \bbZ^d/\calL$ is surjective.
\item For any $m\in \calM$, there exists $w\in \scrW_1$ such that $m + w \not\equiv m' + w' \modu \calL$ for any $m'\in \calM\setminus\{m\}, w'\in \calW\cup \scrW_1$.
\end{enumerate}
\end{theoremIntro}

\begin{remark}
Note that if the set $\calM$ as in Theorem \ref{theorem1.1.3} (resp. Theorem \ref{theorem1.1.4}) contains only one element, then the condition (3) in Theorem \ref{theorem1.1.3} (resp. Theorem \ref{theorem1.1.4}) is vacuously true. 
\end{remark}

Finally, we conclude a general statement on the existence of minimal complements for certain infinite sets in finitely generated abelian groups.

\begin{theoremIntro}
[Theorem \ref{sec5Thm}]
Let $G$ be a finitely generated (infinite) abelian group with the decomposition $G\simeq \mathbb{Z}^{d}\times(\mathbb{Z}/a_{1}\mathbb{Z})\times \cdots \times(\mathbb{Z}/a_{s}\mathbb{Z})$ with $d\geqslant 1$ and where $a_{1}|a_{2}|...|a_{s}$ are positive integers $>1$ (determined uniquely by the isomorphism type of $G$). Suppose $W\subset \mathbb{Z}^{d}$ such that either $W$ satisfies the conditions of Theorem \ref{theorem1.1.3} (equivalently Theorem \ref{Thm: Implies existence of min comple}) or $W$ is a product set of the form $W_{1}\times W_{2}\times \cdots \times W_{d}$ where each $W_{i}\subseteq \mathbb{Z}, \forall 1\leqslant i\leqslant d$ has minimal complements in $\mathbb{Z}$, then $W\times H$ will have a minimal complement in $G$ where $H \subseteq (\mathbb{Z}/a_{1}\mathbb{Z})\times \cdots \times(\mathbb{Z}/a_{s}\mathbb{Z})$ is any arbitrary nonempty subset. 
\end{theoremIntro}

See also Remark \ref{Remark:MinCompEg} for further discussion on more general subsets of finitely generated abelian groups admitting minimal complements.

\section{Existence of minimal complements for finite sets in arbitrary groups}
\label{sec2}

In this section we answer Question \ref{nathansonprob12}.

\begin{proof}[Proof of Theorem \ref{theorem1.1}]
We divide the proof of Theorem \ref{theorem1.1} into two propositions, see Proposition \ref{prop1} (for the countable case) and Proposition \ref{prop2} (for the uncountable case).
\end{proof}

\begin{proposition}
\label{prop1}
Every complement of a nonempty finite subset of an at most countable group contains a minimal complement. 
\end{proposition}

\begin{proof}
For finite groups, the proof of the proposition is clear. Now, let $S$ be a nonempty finite subset of a countably infinite group $G$ and $C$ be a complement of $S$ in $G$. Since $S$ is finite and $G$ is countably infinite, the complement $C$ is countably infinite. Hence the elements of $C$ can be enumerated using the positive integers. Write $C = \{c_i\}_{i=1}^\infty$. Define $C_1$ to be equal to $C$ and for any integer $i\geqslant 1$, define
\begin{equation}
C_{i+1} :=
\begin{cases}
C_i\setminus\{c_i\} & \text{ if $C_i\setminus\{c_i\}$ is a complement to $S$},\\
C_i & \text{ otherwise}. 
\end{cases}
\end{equation}
Since each $C_i$ is a complement to $S$, for each $x\in G$ and any $i\geqslant 1$, there exist elements $c_{x, i}\in C_i, s_{x, i}\in S$ such that $x = s_{x, i}c_{x, i}$. Since $S$ is finite, by the Pigeonhole principle, for some element $s\in S$, the equality $s = s_{x,i}$ holds for infinitely many $i\geqslant 1$. Hence for infinitely many $i$, we obtain $s^\mo x = c_{x, i}$, which is an element of $C_i$. Consequently, for each positive integer $k$, there exists an integer $i_k > k$ such that $s^\mo x$ is an element of $C_{i_k}$, which is contained in $C_k$ (as $i_k > k$) and thus $C_k$ contains $s^\mo x$. Define $M$ to be the intersection $\cap_{i\geqslant 1} C_i$. Then for each $x\in G$, there exists an element $s\in S$ such that $s^\mo x$ belongs to $M$. Hence $M$ is a complement to $S$ in $G$. 

We claim that $M$ is a minimal complement to $S$. On the contrary, assume that $M\setminus \{c_j\}$ is a complement to $S$ for some element $c_j$ in $M$. Since $C_j$ contains $M$, $C_j\setminus \{c_j\}$ is also a complement to $S$. Then $C_{j+1}$ is equal to $C_j\setminus \{c_j\}$. Hence $c_j$ cannot lie in $M$, which is absurd. So $M$ is a minimal complement of $S$ contained in $C$.  
\end{proof}

\begin{proposition}
\label{prop2}
Every complement of a nonempty finite subset of an uncountable group contains a minimal complement. 
\end{proposition}

\begin{proof}
Let $S$ be a finite subset of a group $G$ and $C$ be a complement of $S$ in $G$. Let $\scrC$ denote the set of all complements of $S$ in $G$ which are contained in $C$. Note that $\scrC$ is partially ordered with respect to strict inclusion $\subsetneq$. Note also that the minimal elements of the partially ordered set $(\scrC, \subsetneq)$ are the minimal complements of $S$ in $G$ which are contained in $C$. If every chain in $\scrC$ has a lower bound in $\scrC$, then by Zorn's lemma, it would follow that $\scrC$ has a minimal element. We claim that every chain in $\scrC$ has a lower bound in $\scrC$. Let $\calC = \{C_\lambda\}_{\lambda\in \Lambda}$ be a chain in $\scrC$. If a member of $\calC$ is contained in all other members of $\calC$, then it is a lower bound of $\calC$ in $\scrC$. Henceforth we assume that no member of $\calC$ is a lower bound of $\calC$. 

Note that if $\calC_1, \cdots, \calC_k$ are pairwise disjoint subsets of $\calC$ such that their union is equal to $\calC$, then for some integer $i$ with $1\leqslant i \leqslant k$, each member of $\calC$ contains some member of $\calC_i$. Otherwise, for each $i$, some member $C_{\lambda_i}$ of $\calC$ would be contained in every member of $\calC_i$. Let $j$ be a positive integer $\leqslant k$ such that $C_{\lambda_j}$ is contained in $C_{\lambda_i}$ for any $i$ satisfying $1\leqslant i\leqslant k$. So $C_{\lambda_j}$ is contained in every member of $\calC_i$ for each $i$, which is a contradiction to the assumption that no member of $\calC$ is a lower bound of $\calC$. 

Since $SC_\lambda = G$, for each $x\in G$ and $\lambda\in \Lambda$, there exist elements $s_{x, \lambda}\in S, c_{x,\lambda}\in C_\lambda$ such that $x=s_{x, \lambda} c_{x, \lambda}$. Let $S_x$ denote the subset of elements of $S$ of the form $s_{x, \lambda}$ for some $\lambda\in \Lambda$. For $s\in S_x$, define $\calC_s$ to be the subchain $\{C_\lambda\,|\, s_{x,\lambda} =s\}$ of $\calC = \{C_\lambda\}_{\lambda\in \Lambda}$. Then the sets $\calC_s$, for $s\in S_x$, form a collection of finitely many pairwise disjoint subsets of $\calC$ and their union is equal to $\calC$. Hence by the observation made in the preceding paragraph, there exists an element $s'\in S$ such that each member of $\calC$ contains some element of $\calC_{s'}$. Since $s'^\mo x$ belongs to each member of $\calC_{s'}$, it also belongs to each member of $\calC$ and hence $s'^\mo x$ belongs to the intersection $\cap_{\lambda\in \Lambda} C_\lambda$. Consequently, $\cap_{\lambda\in \Lambda} C_\lambda$ is a complement of $S$ in $G$ contained in $C$. In other words, it is an element of $\scrC$. So each chain in $\scrC$ has a lower bound in $\scrC$. This proves the claim. So the proposition follows from Zorn's lemma.
\end{proof}

\begin{corollary}
 In particular, for finitely generated free abelian groups, e.g., $\mathbb{Z}^k$, every (additive) complement of a finite set contains a minimal complement. 
\end{corollary}

\begin{remark}
Later on, we shall give specific examples of infinite subsets in certain groups which admit minimal complements. See section \ref{Sec:Conclusion}, Propositions \ref{prop6.1} and \ref{prop6.2}. Also see section \ref{Sec: Minimal complement} for minimal complements of eventually periodic sets.
\end{remark}

\subsection{Minimal nets}Let $(X, \delta)$ be a metric space. For $x\in X$  and $r \geqslant 0$,
the sphere with center $x$ and radius $r$ is the set
$$S_{x}(r) := \{z\in X : \delta(x,z) = r \}.$$
The open ball $B_{x}(r)$ of radius $r$ and center $x$ is the set
$$B_{z}(r) := \{ z\in X: \delta(x,z)<r\}$$
while the closed ball $\bar{B}_{z}(r)$ of radius $r$
and center $x$ is,
$$\bar{B}_{z}(r) = \{ z \in X : \delta(x,z) \leqslant r \} .$$

\begin{definition}[$r$-net]
An $r$-net in $(X, d)$ is a subset $N_{r}$ of $X$ with the property that for each element $x\in X$, there exists an element $z \in N_{r}$ such that $\delta(x, z) \leqslant r$. Equivalently, $N_{r}$ is an $r$-net in
$X$ if and only if
$$X =\bigcup_{z\in N_{r}}\bar{B}_{z}(r).$$
\end{definition}

It is clear that $X$ is the unique $0$-net in $X$. The set $N$ is called a net in $(X, \delta)$ if $N$ is an $r$-net in $(X, \delta)$ for some $r \geqslant 0.$ A minimal $r$-net in a metric space $(X, \delta)$
is an $r$-net $N_{r}$ such that no proper subset of $N_{r}$ is an $r$-net in $(X, \delta)$, e.g., in
every metric space $(X, \delta)$, the set $X$ is a minimal $0$-net. As observed by Nathanson in \cite{NathansonMinimalBaseMax}, there is an analogy between minimal asymptotic bases for sets of integers in additive number theory and minimal $r$-nets in metric geometry. In \cite{NathansonAddNT4}, he posed the following question:
\begin{questionIntro}
\cite[Problem 1]{NathansonAddNT4}
``In which metric spaces do there exist minimal $r$-nets for $r > 0?$ When
does a metric space contain an $r$-net $N_{r}$ for some $r > 0$ such that no subset of $N_{r}$ is a minimal $r$-net?"
\end{questionIntro}

In case of groups, a priori there is no metric but one can extend the notion via the metric induced by the generating set. The connection is given by the following lemma:
\begin{lemma}[{\cite[Lemma 2]{NathansonAddNT4}}]
Let $G$ be a group and let $A$ be a symmetric generating set for $G$ with
$1\in A$. For every nonnegative integer $r$, the set $N_{r}$ is an $r$-net in the metric space $(G, \delta_{A})$ if and only if $G = A^{r}N_{r}.$ The set $N$ is a net if and only if $G = A^{r}N$ for
some nonnegative integer $r.$
\end{lemma}
Here $\delta_{A}$ denotes the metric induced from the word length function with respect to the generating set $A$ in $G$. As a corollary of Theorem \ref{theorem1.1}, we deduce

\begin{corollary}
Let $G$ be a finitely generated group. Then a minimal $r$-net always exist for every $r\geqslant 0$.
\end{corollary}
\begin{proof}
Let $G$ be generated by the set $S$ with $|S|< +\infty$. Let $A = S\cup S^{-1}$. Then $|A|< +\infty$ and $A$ is a symmetric generating set of $G$. Fix $r\geqslant 0$, let $N_{r}$ be an $r$-net. Then $G=A^{r}N_{r}$, in particular, $N_{r}$ is a complement of $A^{r}$ in $G$. By Theorem \ref{theorem1.1}, every complement $A^{r}$ has a minimal complement. We are done.
\end{proof}

\section{Inexistence of minimal complements for certain infinite sets}

In the following, $d$ denotes a positive integer and $\bbN$ denotes the set of all nonnegative integers.

\begin{proposition}
\label{prop3.1}
\label{Prop: Complement iff}
Let $M$ be a subset of $\bbZ^d$. Then the following statements are equivalent. 
\begin{enumerate}
\item The set $M$ is a complement of $\bbN^d$ in $\bbZ^d$. 
\item $M$ contains a sequence of elements of $\bbZ^d$ such that the maximum of their coordinates is arbitrarily small negative number.
\item  $\liminf  _{ (x_1, \cdots, x_d)\in M} \max\{x_1, \cdots, x_d\}=-\infty.$
\end{enumerate}
\end{proposition}

\begin{proof}
Suppose statement (1) holds. We need to show that given any integer $n\in \bbZ$, $M$ contains a point all whose coordinates are less than or equal to $n$. On the contrary, suppose this is not true. Then for some $n\in \bbZ$, $M$ is contained in $\bbZ^d \setminus(n-\bbN^d)$. So $M+ \bbN^d$ would also be contained in $\bbZ^d \setminus(n-\bbN^d)$, which is absurd. So the second statement holds. 

It is clear that statement (3) follows from statement (2). 

Suppose statement (3) holds. Then for each integer $n$, the set $M$ contains a point $P_n=(x_{n1}, \cdots, x_{nd})$ with each coordinate $\leqslant n$, i.e., $x_{ni}\leqslant n$ for any $1\leqslant i\leqslant d$. So $\bbN^d$ contains the point $(n-x_{n1}, \cdots, n-x_{nd})$. Hence $\bbN^d$ contains $(n-x_{n1}, \cdots, n-x_{nd}) + \bbN^d$, and thus
$$
P_n + \bbN^d \supseteq P_n + ((n-x_{n1}, \cdots, n-x_{nd}) + \bbN^d)
= 
(n, \cdots, n) + \bbN^d .
$$
Consequently, $M+\bbN^d$ contains $\cup_{n\in \bbZ} ((n, \cdots, n) + \bbN^d) = \bbZ^d$. This proves statement (1). 
\end{proof}

\begin{corollary}
\label{cor3.2}
\label{Cor: Nd minimal complement}
\quad 
\begin{enumerate}
\item The subset $\bbN^d$ has complements in $\bbZ^d$. 
\item The subset $\bbN^d$ has no minimal complement in $\bbZ^d$.
\item No complement of $\bbN^d$ in $\bbZ^d$ contains a minimal complement of $\bbN^d$.
\end{enumerate}
\end{corollary}

\begin{proof}
The first statement follows since $\{(-n,\cdots,-n)\,|\,n\in \bbN \}$ and any of its infinite subsets are complements of $\bbN^d$ in $\bbZ^d$ (which can be seen by applying Proposition \ref{Prop: Complement iff}). For any complement $M$ of $\bbN^d$ and for any finite subset $F$ of $M$, it follows from Proposition \ref{Prop: Complement iff} that $M\setminus F$ is also a complement of $\bbN^d$. Hence the second and the third statement hold.
\end{proof}

We strengthen Proposition \ref{Prop: Complement iff} and Corollary \ref{Cor: Nd minimal complement} in Proposition \ref{Prop: AP Complement iff} below.

\begin{proposition}
\label{prop3.3}
\label{Prop: AP Complement iff}
Let $u_1, \cdots, u_d$ be elements of $\bbZ^d$ which satisfy no nontrivial $\bbZ$-linear relation. Denote the subgroup $\bbZ u_1 + \cdots + \bbZ u_d$ of $\bbZ^d$ by $\calL$. 
\begin{enumerate}
\item A subset $M$ of $\bbZ^d$ is a complement of $\bbN u_1 + \cdots + \bbN u_d$ in $\bbZ^d$ if and only if each fibre of the map $\pi: M \to \bbZ^d/\calL$ contains a sequence of elements of $\bbZ^d$ such that the maximum of their coordinates with respect to $u_1, \cdots, u_d$ is arbitrarily small negative number.
\item The subset $\bbN u_1 + \cdots + \bbN u_d$ in $\bbZ^d$ has complements in $\bbZ^d$. 
\item The subset $\bbN u_1 + \cdots + \bbN u_d$ in $\bbZ^d$ has no minimal complement in $\bbZ^d$.
\item No complement of $\bbN u_1 + \cdots + \bbN u_d$ in $\bbZ^d$ contains a minimal complement of $\bbN u_1 + \cdots + \bbN u_d$ in $\bbZ^d$.
\end{enumerate}
\end{proposition}

\begin{proof}
Let $\frakL$ be a subset of $\bbZ^d$ such that the map $\pi: \frakL \to \bbZ^d/\calL$ is bijective. For any subset $X$ of $\bbZ^d$ and any $\frakl\in \frakL$, let $X_\frakl$ denote the set of all elements of $X$ which are congruent to $\frakl\modu \calL$. Since each element of $\bbN u_1 + \cdots + \bbN u_d$ is congruent to zero modulo $\calL$, it follows that a subset $M$ of $\bbZ^d$ is a complement of $\bbN u_1 + \cdots + \bbN u_d$ in $\bbZ^d$ if and only if $M_\frakl + (\bbN u_1 + \cdots + \bbN u_d) = \bbZ^d_\frakl$ for any $\frakl\in \frakL$. 

Choose an element $\frakl\in \frakL$. Suppose $M_\frakl + (\bbN u_1 + \cdots + \bbN u_d) = \bbZ^d_\frakl$ holds. We claim that $M_\frakl$ contains a sequence of elements of $\bbZ^d_\frakl$ such that the maximum of their coordinates with respect to $u_1, \cdots, u_d$ is arbitrarily small negative number. On the contrary, suppose this is false. Then for some $n\in \bbZ$, 
$$M_\frakl \subseteq \bbZ^d_\frakl \setminus(\frakl + (n(u_1 + \cdots + u_n))-(\bbN u_1 + \cdots + \bbN u_d)).$$
So $M_\frakl+ ( \bbN u_1 + \cdots + \bbN u_d)$ would also be contained in $\bbZ^d_\frakl \setminus(\frakl + (n(u_1 + \cdots + u_n))-(\bbN u_1 + \cdots + \bbN u_d))$, which is absurd. So the claim follows. Conversely, assume that $M_\frakl$ contains a sequence of elements of $\bbZ^d_\frakl$ such that the maximum of their coordinates with respect to $u_1, \cdots, u_d$ is arbitrarily small negative number. So for any $n\in \bbZ$, the set $M_\frakl$ contains a point $P_n=\frakl + x_{n1}u_1 + \cdots +  x_{nd}u_d$ with $x_{ni} \leqslant n$ for any $1\leqslant i\leqslant d$. So $\bbN u_1 + \cdots + \bbN u_d$ contains the point $(n-x_{n1})u_1 + \cdots + ( n-x_{nd})u_d$. Hence $\bbN u_1 + \cdots + \bbN u_d$ contains $((n-x_{n1})u_1 + \cdots + ( n-x_{nd})u_d) + (\bbN u_1 + \cdots + \bbN u_d)$, and thus
\begin{align*}
P_n + \bbN u_1 + \cdots + \bbN u_d 
&\supseteq P_n + (((n-x_{n1})u_1 + \cdots + ( n-x_{nd})u_d)\\
& \quad + (\bbN u_1 + \cdots + \bbN u_d))\\
&=\frakl + (n(u_1 + \cdots + u_n)) + (\bbN u_1 + \cdots + \bbN u_d) .
\end{align*}
Consequently, $M_\frakl+ (\bbN u_1 + \cdots + \bbN u_d)$ contains $\cup_{n\in \bbZ} (\frakl + (n(u_1 + \cdots + u_n)) + (\bbN u_1 + \cdots + \bbN u_d) ) = \bbZ^d_\frakl$. This proves statement (1). 

From the first statement, it follows that the set $\frakL + (-\bbN(u_1+\cdots + u_d))$ and any of its subsets $C$ containing infinitely many elements of $\frakl + (-\bbN(u_1+\cdots + u_d))$ for any $\frakl\in \frakL$ is a complement to $\bbN u_1 + \cdots + \bbN u_d$. So the second statement holds. 

For any complement $M$ of $\bbN u_1 + \cdots + \bbN u_d$ in $\bbZ^d$ and any finite subset $F$ of $M$, it follows from the first statement that $M\setminus F$ is also a complement to $\bbN u_1 + \cdots + \bbN u_d$. This proves the third and the fourth statement. \end{proof}

\section{Minimal complements of infinite sets in free abelian groups}
\label{Sec: Minimal complement}
We start the section with the statement on existence of minimal complements of subgroups of arbitrary groups as stated in the introduction.

\begin{proposition}
\label{prop4.1}
Let $G$ be an arbitrary group with $S$ a subgroup of $G$. Then every complement of $S$ in $G$ has a minimal complement.
\end{proposition}
\begin{proof}
Let $C$ be a complement of $S$ in $G$. Let $M$ be a subset of $C$ such that $m_1m_2^\mo\notin S$ for any two distinct elements $m_1, m_2$ of $M$, and for any $c\in C$, $S$ contains $m_c c^\mo$ for some $m_c\in M$. Then $M$ is a minimal complement to $S$ contained in $C$.
\end{proof}

With the subgroup case done, we shall study minimal complements of sets which are not subgroups.

\begin{lemma}
\label{Lemma: coset one}
Let $S$ be a subset of a group $G$. Suppose there exists a subgroup $H$ of $G$ such that the composite map $S \to G \to G/H$ is surjective and at least one fibre of this composite map contains exactly one element. Then $H$ is a minimal complement of $S$ in $G$. 
\end{lemma}

In the above, $G/H$ denotes the set of equivalence classes under the equivalence relation $\sim$ defined on $G$ by $g_1\sim g_2$ if $g_1^\mo g_2\in H$. 

\begin{proof}
Note that $H$ is a complement of $S$ in $G$ if and only if the composite map $S \to G \to G/H$ is surjective. If at least one fibre of the surjective map $S \to G \to G/H$ contains exactly one element, then no nonempty proper subset of $H$ could be a complement to $S$. So the Lemma follows.
\end{proof}

Next, we will consider free abelian groups of finite rank. For notational convenience, we will consider the free abelian group $\bbZ^d$ only. Note that a condition holds for sufficiently large elements of $\bbZ^d$ if and only if the condition holds for almost all elements of $\bbZ^d$. This enables us to use the terms ``sufficiently large'' and ``almost all'' interchangeably without any confusion. However, to refer to sufficiently large elements, the underlying space is certainly required to have a notion of a metric (which in general is not available in an arbitrary free abelian group). 

We will prove Theorems \ref{Thm: existence of min complement implies}, \ref{Thm: Implies existence of min comple} providing a necessary and a sufficient condition for the existence of minimal complements for eventually periodic sets. We also provide a necessary and sufficient condition for certain eventually periodic sets to have a minimal complement. See Theorem \ref{Thm: Even peri 1 has min comp iff}.

\subsection{Periodic and eventually periodic sets}

The union of two subsets $A, B$ of $\bbZ^d$ is denoted by $A\cup B$. When $A, B$ are disjoint, we write $A\sqcup B$ to denote the union $A\cup B$ and to indicate that $A, B$ are disjoint. 

In the following, $u_1, \cdots, u_d$ denote elements of $\bbZ^d$ which satisfy no nontrivial $\bbZ$-linear relation. Let $\calL$ denote the subgroup $\bbZ u_1 + \cdots +  \bbZ u_d$ of $\bbZ^d$. Two elements $x,y\in \bbZ^d$ are said to be \textit{equivalent modulo $\calL$} if $x-y\in \calL$. Denote by $\pi$ the quotient map $\pi: \bbZ^d \twoheadrightarrow \bbZ^d/\calL$.  The image of an element $v\in \bbZ^d$ under the quotient map $\pi: \bbZ^d \twoheadrightarrow \bbZ^d/\calL$ is denoted by $\bar v$. A typical element of $\bbZ^d/\calL$ would be denoted by $\bar v$ (which is legitimate since the quotient map $\pi: \bbZ^d \twoheadrightarrow \bbZ^d/\calL$ is surjective), and by $v$, we would denote a lift of $\bar v$ to $\bbZ^d$. For any subset $X$ of $\bbZ^d$ and any $\bar v\in \bbZ^d/\calL$, denote by $X_{\bar v}$ the intersection $X\cap \pi^\mo (\bar v)$. For any subset $X$ of $\bbZ^d$, the restriction of the map $\pi$ to $X$ is the composite map $X\hra \bbZ^d \xra{\pi} \bbZ^d/\calL$, which we denote by $\pi_X$ (or, simply by $\pi$ to avoid cumbersome notation). 

\begin{definition}
\label{Defn: periodic}
A nonempty subset $X$ of $\bbZ^d$ is called \textnormal{periodic with periods} $u_1,\cdots, u_d$ if $X$ is contained in $F + \bbN u_1 + \cdots + \bbN u_d$ for some nonempty finite subset $F\subset \bbZ^d$ and $x + \bbN u_1 + \cdots + \bbN u_d$ is contained in $X$ for any $x\in X$. A nonempty subset $X$ of $\bbZ^d$ is called \textnormal{eventually periodic with periods} $u_1,\cdots, u_d$ if $X$ is contained in $F+ \bbN u_1 + \cdots + \bbN u_d$ for some nonempty finite subset $F\subset \bbZ^d$ and $x + \bbN u_1 + \cdots + \bbN u_d$ is contained in $X$ for any sufficiently large element $x\in X$. 
\end{definition}

\begin{example}
For any integer $k\geqslant 0$, the set $\{(r, -r)\in \bbZ^2\,|\, -k\leqslant r\leqslant k\} + \bbN^2$ is a periodic subset of $\bbZ^2$ with periods $(1,0), (0,1)$. The set $\bbN u_1 + \cdots + \bbN u_d$ appearing in Proposition \ref{Prop: AP Complement iff} is periodic and also eventually periodic. Note that the periodic subsets of $\bbZ^d$ are also eventually periodic with the same periods. The set $\bbZ^d$ is neither periodic nor eventually periodic.  
\end{example}

First we prove Proposition \ref{Prop: Structure of Periodic} and Theorem \ref{Thm: Structure of Eventually Periodic}, which describe the structure of periodic sets and eventually periodic sets. 

\begin{proposition}
\label{Prop: Structure of Periodic}
Let $W$ be a periodic subset of $\bbZ^d$ with periods $u_1, \cdots, u_d$. Let $\calQ$ denote the image of $W$ under the map $\pi: W\to \bbZ^d/\calL$. 
\begin{enumerate}
\item 
The nonempty fibres of the map $\pi: W \to \bbZ^d/\calL$, i.e., the sets $W_{\bar v}$ for $\bar v\in \calQ$ are periodic with periods $u_1, \cdots, u_d$. 
\item 
For any $\bar v\in \calQ$, denote by $\calW'_{\bar v}$ the set of elements $w\in W_{\bar v}$ such that the intersection of the sets $w - (\bbN u_1 + \cdots + \bbN u_d)$ and $W_{\bar v}$ contains no element other than $w$. Then for any $\bar v\in \calQ$, the set $\calW'_{\bar v}$ is a nonempty finite subset of $W_{\bar v}$ and the equality
\begin{equation}
W_{\bar v} = \calW'_{\bar v} + (\bbN u_1 + \cdots + \bbN u_d) 
\end{equation}
holds.
\item 
Let $\calW$ denote the set of elements $w\in W$ such that the intersection of the sets $w- (\bbN u_1 + \cdots + \bbN u_d)$ and $W$ contains no element other than $w$. Then the equalities 
\begin{align*}
\calW &= \sqcup_{\bar v\in \calQ} \calW'_{\bar v}, \\
\calW_{\bar v}& = \calW'_{\bar v}, \\ 
W 
&= \calW + (\bbN u_1 + \cdots + \bbN u_d)\\
&= \sqcup _{\bar v\in \calQ}  ( \calW_{\bar v} + (\bbN u_1 + \cdots + \bbN u_d ))  
\end{align*}
hold.
\end{enumerate}
\end{proposition}

\begin{proof}
Since $W$ is periodic with periods $u_1, \cdots, u_d$, and $\calL$ is equal to $\bbZ u_1 + \cdots + \bbZ u_d$, it follows that the nonempty fibres of the map $\pi: W\to \bbZ^d/\calL$ are periodic with periods $u_1, \cdots, u_d$. These fibres are precisely the sets $W_{\bar v}$ for $\bar v\in \calQ$. This proves the first statement.

Let $\bar v$ be an element of $\calQ$ and $v\in W_{\bar v}$ be a lift of $\bar v$ modulo $\calL$. Note that $-v+ W_{\bar v}$ is periodic with periods $u_1, \cdots, u_d$ and there exists a nonempty finite subset $F$ of $\bbZ^d$ such that $-v+ W_{\bar v}$ is contained in $F+ \bbN u_1 + \cdots + \bbN u_d$. Since $-v+ W_{\bar v}$ is contained in $\bbZ u_1 + \cdots + \bbZ u_d$, note that $F$ can be taken to be a subset of $\bbZ u_1 + \cdots + \bbZ u_d$. For $1\leqslant i\leqslant d$, let $\lambda_i$ denote the minimum of the $i$-th coordinates of the elements of $F$. Then it follows that $F+ \bbN u_1 + \cdots + \bbN u_d$ is contained in $(\lambda_1 u_1 + \cdots + \lambda_d u_d)+(\bbN u_1 + \cdots + \bbN u_d)$. So $-v+ W_{\bar v}$ is periodic with periods $u_1, \cdots, u_d$ and it is contained in $(\lambda_1 u_1 + \cdots + \lambda_d u_d)+(\bbN u_1 + \cdots + \bbN u_d)$. Note that 
\begin{equation}
-v+\calW'_{\bar v} = 
\{u\in (-v+ W_{\bar v})\,|\, (u-(\bbN u_1 + \cdots + \bbN u_d))\cap (-v+ W_{\bar v}) = \{u\} \}.
\end{equation}
Then the second statement follows since if a periodic subset $U$ of $\bbZ^d$ with periods $u_1, \cdots, u_d$ is contained in $(\mu_1 u_1 + \cdots + \mu_d u_d)+(\bbN u_1 + \cdots + \bbN u_d)$ for some integers $\mu_1, \cdots, \mu_d$, then $\calU := \{u\in U \,|\, (u-(\bbN u_1 + \cdots + \bbN u_d))\cap U = \{u\} \}$ is a nonempty finite subset of $U$ and 
\begin{equation}
U = \calU + (\bbN u_1 + \cdots + \bbN u_d)
\end{equation}
holds.

If $w$ is an element of $\calW$ and $\pi(w)$ is equal to $\bar v$, then the intersection of the sets $w - (\bbN u_1 + \cdots + \bbN u_d)$ and $W_{\bar v}$ contains no element other than $w$. Hence $w$ belongs to $\calW'_{\bar v}$. So $\calW$ is contained in $\sqcup_{\bar v\in \calQ} \calW'_{\bar v}$. Note that for any $\bar v\in \calQ$ and any $w\in W_{\bar v}$, 
\begin{equation}
(w - (\bbN u_1 + \cdots + \bbN u_d) )\cap W = 
(w - (\bbN u_1 + \cdots + \bbN u_d) )\cap W_{\bar v} = \{w\}.
\end{equation}
This implies that $\sqcup_{\bar v\in \calQ} \calW'_{\bar v}$ is contained in $\calW$, proving the first equality in statement (3). Then the remaining two equalities are immediate.
\end{proof}

\begin{theorem}
\label{Thm: Structure of Eventually Periodic}
Let $W$ be an eventually periodic subset of $\bbZ^d$ with periods $u_1, \cdots, u_d$. 
\begin{enumerate}
\item 
Let $\scrW$ denote the set of all elements $w\in W$ such that $w+ \bbN u_1 + \cdots + \bbN u_d$ is not contained in $W$. Then $\scrW$ is a finite set. 
\item 
The set $W\setminus \scrW$ is periodic with periods $u_1, \cdots,u_d$.
 \item 
The following sets are equal. 
\begin{enumerate}
\item The set of all elements of $\bbZ^d/\calL$ having infinite fibre under the composite map $W\hra \bbZ^d \twoheadrightarrow \bbZ^d/\calL$.
\item The set of all elements of $\bbZ^d/\calL$ having infinite fibre under the composite map $W\setminus \scrW\hra \bbZ^d \twoheadrightarrow \bbZ^d/\calL$.
\item The image $\pi(W\setminus \scrW)$ of $W\setminus \scrW$ under $\pi$, to be denoted by $\calQ$. 
\end{enumerate}
\item 
The nonempty fibres of the map $\pi: W\setminus \scrW\to \bbZ^d/\calL$, i.e., the sets $(W\setminus \scrW)_{\bar v}$ for $\bar v\in \calQ$ are periodic with periods $u_1, \cdots, u_d$. 
\item 
For any $\bar v\in \calQ$, denote by $\calW'_{\bar v}$ the set of elements $w\in (W\setminus\scrW)_{\bar v}$ such that the intersection of the sets $w - (\bbN u_1 + \cdots + \bbN u_d)$ and $(W\setminus\scrW)_{\bar v}$ contains no element other than $w$. Then for any $\bar v\in \calQ$, the equality
\begin{equation}
(W\setminus\scrW)_{\bar v} = \calW'_{\bar v} + (\bbN u_1 + \cdots + \bbN u_d) 
\end{equation}
holds. Moreover, $\calW'_{\bar v}$ is a finite set.
\item 
Let $\calW$ denote the set of elements $w\in W\setminus \scrW$ such that the intersection of the sets $w- (\bbN u_1 + \cdots + \bbN u_d)$ and $W\setminus \scrW$ contains no element other than $w$. Then the equalities 
\begin{align*}
\calW & = \sqcup_{\bar v\in \calQ} \calW'_{\bar v}, \\
\calW_{\bar v} & = \calW'_{\bar v}, \\ 
W \setminus \scrW
&= \calW + (\bbN u_1 + \cdots + \bbN u_d)\\
&= \sqcup _{\bar v\in \calQ}  ( \calW_{\bar v} + (\bbN u_1 + \cdots + \bbN u_d ))  
\end{align*}
hold.
\item
The following 
\begin{equation}
\label{Eqn: Decomposition}
W 
= \scrW \sqcup (\calW + (\bbN u_1 + \cdots + \bbN u_d))
= \mathscr W \sqcup ( \sqcup _{\bar v\in \calQ}  ( \calW_{\bar v} + (\bbN u_1 + \cdots + \bbN u_d ))  )
\end{equation}
expresses $W$ as a disjoint union of its subsets.
\end{enumerate}
\end{theorem}

\begin{proof}
Since $W$ is eventually periodic with periods $u_1, \cdots, u_d$, it contains $w + \bbN u_1 + \cdots +\bbN u_d$ for almost all $w\in W$. So it follows that $\scrW$ is finite. 

Since $W$ is nonempty, it is infinite. So $W\setminus\scrW$ is nonempty. Moreover, since $W$ is eventually periodic, it follows that $W\setminus\scrW$ is contained in $F + (\bbN u_1 + \cdots + \bbN u_d)$ for some finite subset $F$ of $\bbZ^d$. By the definition of $\scrW$, it follows that $(W\setminus \scrW) + (\bbN u_1 + \cdots + \bbN u_d)$ is contained in $W$. To prove that $W\setminus \scrW$ is periodic with periods $u_1, \cdots, u_d$, it only remains to show that $(W\setminus \scrW) + (\bbN u_1 + \cdots + \bbN u_d)$ does not intersect with $\scrW$. Suppose $\scrW$ contains $w+ \lambda_1 u_1 + \cdots + \lambda_d u_d$ for some $w\in W\setminus \scrW$ and for some $\lambda_1, \cdots, \lambda_d\in \bbN$. So for some $\mu_1, \cdots, \mu_d\in \bbN$, the vector $(w + \lambda_1 u_1 + \cdots + \lambda_d u_d) + (\mu_1 u_1 + \cdots + \mu_d u_d)$ would not be contained in $W$, which is absurd since $w\in W\setminus \scrW$. Consequently, $(W\setminus \scrW) + ( \bbN u_1 + \cdots + \bbN u_d)$ is contained in $W\setminus \scrW$. Hence $W\setminus \scrW$ is periodic with periods $u_1, \cdots,u_d$.

Note $W= \scrW\sqcup (W\setminus \scrW)$ is a decomposition of $W$ into two disjoint subsets where $\scrW$ is finite and $W\setminus \scrW$ is infinite. Since the set $\bbZ^d/\calL$ is finite, it follows that an element $\bar v$ of $\bbZ^d/\calL$ has infinite fibre under $\pi: W \to \bbZ^d/\calL$ if and only if it has infinite fibre under $\pi: W\setminus \scrW \to \bbZ^d/\calL$. Moreover, any element in the image of the map $\pi: W\setminus \scrW \to \bbZ^d/\calL$ has infinite fibre because $W\setminus \scrW$ contains the $(\bbN u_1 + \cdots + \bbN u_d)$-translate of any of its points. Hence the three sets in part (a), (b), (c) are equal. 

The statements (4), (5), (6), (7) follow from Proposition \ref{Prop: Structure of Periodic}. 
\end{proof}

\begin{remark}
The decomposition of an eventually periodic subset $W$ of $\bbZ^d$ with periods $u_1, \cdots, u_d$ as 
$$W 
= \scrW \sqcup (\calW + (\bbN u_1 + \cdots + \bbN u_d))
$$
is unique 
(where $\scrW, \calW$ are as in Theorem \ref{Thm: Structure of Eventually Periodic}) in the sense that if $\scrW', \calW'$ are subsets of $W$ such that $w + (\bbN u_1 + \cdots + \bbN u_d)$ is not contained in $W$ for any $w\in \scrW'$ and $W$ is equal to the disjoint union $\scrW' \sqcup (\calW' + (\bbN u_1 + \cdots + \bbN u_d))$, then $\scrW = \scrW', \calW = \calW'$. 
\end{remark}

\begin{remark}
Note that not all the nonempty subsets $X$ of $\bbZ^d$ satisfying 
\begin{equation}
\label{Eqn: Periodic}
X\supseteq X +(\bbN u_1 + \cdots + \bbN u_d)
\end{equation}
 are of the form $\calW + (\bbN u_1 + \cdots + \bbN u_d)$ for some subset $\calW$ of $\bbZ^d$ satisfying condition (3) of Proposition \ref{Prop: Structure of Periodic} (for instance, consider the subset $\bbZ\times \bbN$ of $\bbZ^2$). Moreover, there are subsets $X$ of $\bbZ^d$ which satisfy Equation \eqref{Eqn: Periodic} and are of the form $\calW  + (\bbN u_1 + \cdots + \bbN u_d)$, but are not periodic, i.e., not contained in the sum of $\bbN u_1 + \cdots + \bbN u_d$ and a nonempty finite set (for instance, consider the subset $\{ (n,-n)\,|\, n\in \bbN\} + \bbN^2$ of $\bbZ^2$). 
\end{remark}

\subsection{Necessary condition}

\begin{theorem}
\label{Thm: existence of min complement implies}
Let $W$ be an eventually periodic subset of $\bbZ^d$ with periods $u_1, \cdots, u_d$. Let $\scrW, \calW$ be as in Theorem \ref{Thm: Structure of Eventually Periodic}. Let $M$ be a minimal complement of $W$ in $\bbZ^d$. Let $M_\infty$ denote the union of those fibres of the map $\pi: M \to \bbZ^d/\calL$ containing a sequence of elements of $\bbZ^d$ such that the maximum of their coordinates with respect to $u_1, \cdots, u_d$ is arbitrarily small negative number. Let $\calM$ be a subset of $M_\infty$ such that the composite map $\calM\hra M_\infty \twoheadrightarrow \pi(M_\infty)$ is a bijection. Let $\scrW_0$ (resp. $\scrW_1$) denote the set of elements of $\scrW$ which are congruent to some element (resp. no element) of $\calW$ modulo $\calL$\footnote{
Equivalently, $\scrW_0$ is the inverse image of $\calQ$ under the map $\pi: \scrW\to \bbZ^d/\calL$, and $\scrW_1$ is its complement in $\scrW$.
}. Then the following statements hold.
\begin{enumerate}
\item The set $M_\infty$ is an infinite set, $\calM$ is a nonempty finite subset of $M_\infty$, $\scrW_1$ is nonempty, and the map $\pi: (\calM + (\calW \cup \scrW_1)) \to \bbZ^d/\calL$ is surjective.
\item For any $m\in \calM$, there exists $w\in \scrW_1$ such that $m + w \not\equiv m' + w' \modu \calL$ for any $m'\in \calM, w'\in \calW$.
\end{enumerate}
\end{theorem}

\begin{proof}
Let $M_\fin$ denote $M\setminus M_\infty$. Let $\frakL$ be a subset of $\bbZ^d$ such that the map $\pi: \frakL \to \bbZ^d/\calL$ is bijective. By Theorem \ref{Thm: Structure of Eventually Periodic} (7), it follows that for each $\frakl\in \frakL$, there exists a positive integer $\lambda_\frakl$ such that the intersection of $M_\fin + W$ and $\frakl +(-\lambda_\frakl (u_1 + \cdots + u_d)) +(-(\bbN u_1 + \cdots + \bbN u_d))$ is empty. 
Hence $M_\fin +W$ contains no element of the set $\frakL + (-(\max_{\frakl\in \frakL}\lambda_\frakl) (u_1 + \cdots + u_d)) +(-(\bbN u_1 + \cdots + \bbN u_d))$. 
Since $M$ is a complement of $W$ in $\bbZ^d$ and $M$ is equal to the union of $M_\fin$ and $M_\infty$, it follows that $M_\infty$ is nonempty and 
\begin{equation}
\label{Eqn: Containment}
M_\infty + W \supseteq \frakL + (-(\max_{\frakl\in \frakL}\lambda_\frakl (u_1 + \cdots + u_d))) +(-(\bbN u_1 + \cdots + \bbN u_d)).
\end{equation}
So $M_\infty$ is infinite and $\calM$ is nonempty. If $\scrW$ were empty, then Equation \eqref{Eqn: Containment} would imply 
\begin{equation}
\label{Eqn: Containment 2}
M_\infty + \calW + (\bbN u_1 + \cdots + \bbN u_d) \supseteq \frakL + (-(\max_{\frakl\in \frakL}\lambda_\frakl (u_1 + \cdots + u_d))) +(-(\bbN u_1 + \cdots + \bbN u_d)).
\end{equation}
Then for any $m_0\in M_\infty$, $M\setminus\{m_0\}$ would be a complement to $W$. Indeed, given an element $x\in \bbZ^d$, it is equal to $m+w$ for some $m\in M$ and $w\in W$. If $m\neq m_0$, then $x$ belongs to $M\setminus\{m_0\}+W$. If $m=m_0$, then for some positive integers $\mu_1, \cdots, \mu_d$, $M_\infty$ would contain $m_0-(\mu_1 u_1 + \cdots + \mu_d u_d)$. Since $w$ belongs to $W = \calW + (\bbN u_1 + \cdots + \bbN u_d)$, it follows that $w + (\mu_1 u_1 + \cdots + \mu_d u_d)$ belongs to $W$. Hence $x$ belongs to $M\setminus\{m_0\}+W$. So $M\setminus\{m_0\}+W$ is equal to $\bbZ^d$, which is absurd. Hence $\scrW$ is nonempty. Now, using Theorem \ref{Thm: Structure of Eventually Periodic} (7) again, we obtain from Equation \eqref{Eqn: Containment} that the map 
$$\pi: (\calM + (\calW \cup \scrW)) \to \bbZ^d/\calL$$
is surjective. By the definition of $\scrW_0$, it follows that the map $\pi: (\calM + (\calW \cup \scrW_1)) \to \bbZ^d/\calL$ is surjective. This proves the first statement.

Now we prove that the second statement is true. On the contrary, suppose the second statement is false. So there exists an element $m\in \calM$ such that for each $w\in \scrW_1$, there exist elements $m'\in \calM, w'\in \calW$ such that $m+w \equiv m'+w'\modu \calL$. We prove that $M\setminus \{m\}$ is a complement to $W$, which would contradict the minimality of $M$, and thereby establish the second statement. 

Let $x$ be an element of $\bbZ^d$. Since $M$ is a complement of $W$, it follows that $x = m_0 + w_0$ for some $m_0\in M, w_0\in W$. If $m_0\neq m$, then $x$ belongs to $M\setminus\{m\} + W$. Suppose $m_0$ is equal to $m$, i.e., $x = m+w_0$. If $w_0$ belongs to $\scrW_1$, then there exist elements $m'\in \calM, w'\in \calW$ such that $m+w_0\equiv m'+w'\modu \calL$. This gives $x = m'+w'+(n_1u_1 + \cdots + n_du_d)$ for some integers $n_1, \cdots, n_d$. Since $w'$ belongs to $\calW$, for any integers $r_1, \cdots, r_d$ with $r_i\geqslant - n_i$, $W$ contains $w' + (n_1u_1 + \cdots + n_du_d) + (r_1u_1 + \cdots + r_du_d)$. Since $m'\in \calM\subseteq M_\infty$, it follows that $m'-(\lambda_1 u_1 + \cdots + \lambda_d u_d)\in M_\infty \setminus\{m\}\subseteq M\setminus\{m\}$ for some integers $\lambda_1, \cdots, \lambda_d$ satisfying $\lambda_i\geqslant \max \{0, -n_i\}$ for $1\leqslant i\leqslant d$. Hence $x = (m'-(\lambda_1 u_1 + \cdots + \lambda_d u_d)) + (w'+ (n_1u_1 + \cdots + n_du_d) +( \lambda_1 u_1 + \cdots + \lambda_d u_d) )$ belongs to $M \setminus \{m\} + W$. On the other hand, if $w_0$ does not belong to $\scrW_1$, then $w_0$ belongs to $\scrW_0 \sqcup (\calW + (\bbN u_1 + \cdots + \bbN u_d))$. Since the image of $\scrW_0$ under $\pi: \scrW_0\to \bbZ^d/\calL$ is contained in $\calQ$ and the image of $\calW$ under $\pi: \calW \to \bbZ^d/\calL$ is equal to $\calQ$, there exist positive integers $\mu_1, \cdots, \mu_d$ such that $w_0 + \mu_1 u_1 + \cdots + \mu_d u_d$ belongs to $\calW + (\bbN u_1 + \cdots + \bbN u_d)$. Moreover, since $m\in M\subseteq M_\infty$, there exist integers $\lambda_1, \cdots, \lambda_d$ with $\lambda_i\geqslant \mu_i$ for any $1\leqslant i\leqslant d$ such that $m - (\lambda_1 u_1 + \cdots + \lambda_d u_d)$ belongs to $M_\infty \setminus \{m\} \subseteq M\setminus\{m\}$. Also note that $w_0 + \lambda_1 u_1 + \cdots + \lambda_d u_d$ belongs to $W$ because $w_0 + \mu_1 u_1 + \cdots + \mu_d u_d$ belongs to $\calW + (\bbN u_1 + \cdots + \bbN u_d)\subseteq W$ and $\lambda_i\geqslant \mu_i$ for any $1\leqslant i\leqslant r$. So $x = m+w_0$ belongs to $M\setminus\{m\} + W$. Consequently, $x$ belongs to $M\setminus\{m\} + W$ whether $w_0$ belongs to $\scrW_1$ or not. Hence $M\setminus\{m\}$ is a complement to $W$, which is a contradiction to the given condition that $M$ is a minimal complement to $W$. So the second statement is true. 
\end{proof}

\begin{remark}
Note that Theorem \ref{Thm: existence of min complement implies} is in accordance with Corollary \ref{Cor: Nd minimal complement}(2), (3) and Proposition \ref{Prop: AP Complement iff}(3), (4). 
\end{remark}

\begin{corollary}
\label{Cor: No periodic has minimal}
By Theorem \ref{Thm: existence of min complement implies}(1), periodic subsets of $\bbZ^d$ do not have minimal complements in $\bbZ^d$.
\end{corollary}

\subsection{Sufficient condition}

\begin{theorem}
\label{Thm: Implies existence of min comple}
Let $W$ be an eventually periodic subset of $\bbZ^d$ with periods $u_1, \cdots, u_d$. Let $\scrW, \calW$ be as in Theorem \ref{Thm: Structure of Eventually Periodic}, and $\scrW_1$ be as in Theorem \ref{Thm: existence of min complement implies}. Suppose $\scrW_1$ is nonempty and there exists a nonempty finite subset $\calM$ of $\bbZ^d$ such that the following conditions hold. 
\begin{enumerate}
\item The map $\pi: \calM\to \bbZ^d/\calL$ is injective.
\item The map $\pi: (\calM + (\calW \cup \scrW_1)) \to \bbZ^d/\calL$ is surjective.
\item For any $m\in \calM$, there exists $w\in \scrW_1$ such that $m + w \not\equiv m' + w' \modu \calL$ for any $m'\in \calM\setminus\{m\}, w'\in \calW\cup \scrW_1$.
\end{enumerate}
Then $W$ has a minimal complement in $\bbZ^d$.
\end{theorem}

\begin{proof}
Let $C$ denote the collection of all elements in $\bbZ^d$ which are congruent to some element of $\calM$ modulo $\calL$, and $C'$ denote the collection of all elements in $\bbZ^d$ which are congruent to no element of $\calM + \calW$ modulo $\calL$. 
Note that $C+\scrW_1$ contains all elements of $\bbZ^d$ which are congruent to some element of $\calM+\scrW_1$. By the first condition, it follows that each element of $C'$ is congruent to some element of $\calM+\scrW_1$. So $C+\scrW_1$ contains $C'$. 

We claim that $C$ has a subset $M$ such that $M$ is minimal among the subsets of $C$ with respect to the property that $M+\scrW_1$ contains $C'$. Note that $C$ is countably infinite. Hence its elements can be enumerated by the positive integers. Write $C=\{c_i\,|\, i\geqslant 1\}$. Define $C_1$ to be equal to $C$ and for any integer $i\geqslant 1$, define 
\begin{equation}
C_{i+1}=
\begin{cases}
C_i\setminus\{c_i\} & \text{ if $C_i\setminus\{c_i\}+\scrW_1$ contains }C',\\
C_i & \text{ otherwise}.
\end{cases}
\end{equation}
Let $M$ denote the intersection $\cap_{i\geqslant 1} C_i$. Note that for each $x\in C'$ and any integer $i\geqslant 1$, there exist elements $c_{x,i}\in C_i$ and $w_{x,i}\in \scrW_1$ such that $x = c_{x,i} + w_{x,i}$ holds. Since $\scrW_1$ is a nonempty finite set, by the Pigeonhole principle, for some element $t\in \scrW_1$, the equality $t=w_{x,i}$ holds for infinitely many positive integers $i$. Hence for infinitely many $i\geqslant 1$, we obtain $x-t=c_{x,i}$, which is an element of $C_i$. Consequently, for each positive integer $k$, there exists an integer $i_k > k$ such that $x-t$ is an element of $C_{i_k}$, which is contained in $C_k$ (as $i_k > k$) and thus $C_k$ contains $x-t$. We conclude that for each $x\in C'$, there exists an element $t\in \scrW_1$ such that $x-t$ belongs to $C_k$ for any $k\geqslant 1$, i.e., $x-t\in M$. In other words, $M+\scrW_1$ contains $C'$. Moreover, $M$ is minimal with respect to the property that $M+\scrW_1$ contains $C'$. On the contrary, assume that for some integer $j\geqslant 1$, $M$ contains $c_j$ and $M\setminus\{c_j\}+\scrW_1$ contains $C'$. Since $M$ is contained in $C_j$, it follows that $C_j\setminus\{c_j\}+\scrW_1$ contains $C'$. So $C_{j+1}$ does not contain $c_j$ and hence $c_j$ does not belong to $M= \cap _{i\geqslant 1} C_i$, which is absurd. This proves the claim. 

We claim that $M$ is a minimal additive complement to $S$ in $\bbZ^d$. For $m\in \calM$, define $M_m$ to be the set of all elements of $M$ congruent to $m$ modulo $\calL$, i.e., $M_m$ is the fibre of the map $M\to \bbZ^d/\calL$ at the point $\bar m$. Note that for any $m\in \calM$, the set $M_m$ contains a sequence of points such that the maximum of their coordinates with respect to $u_1, \cdots, u_d$ is arbitrary small negative number. Indeed, if we fix $m\in \calM$, then note that by condition (3), there exists $w\in \scrW_1$ such that $m+w \not \equiv m' + w'\modu \calL$ for any $m'\in \calM\setminus\{m\}$ and $w'\in \calW\cup \scrW_1$. Then it follows that $m+w$ is congruent to no element of $\calM+\calW$ modulo $\calL$, and hence $m+w+\calL$ is contained in 
$$C'\subseteq M+\scrW_1=(M_m+\scrW_1)\sqcup (\sqcup_{m'\in \calM \setminus\{m\}} (M_{m'}+\scrW_1)).$$ Furthermore, condition (3) implies that $\sqcup_{m'\in \calM \setminus\{m\}} (M_{m'}+\scrW_1)$ contains no element of $m+w+ \calL$. So $m+w+\calL$ is contained in $M_m+\scrW_1$. Thus 
$$m+w+(-\bbN (u_1 + \cdots +  u_d))\subseteq M_m+\scrW_1=\cup_{w'\in \scrW_1}(w'+M_m).$$ 
Since $\scrW_1$ is finite, it follows that for some $w'\in \scrW_1$, $w'+M_m$ contains a sequence of points of $m+w+(-\bbN (u_1 + \cdots + u_d))$ such that the maximum of their coordinates with respect to $u_1, \cdots, u_d$ is arbitrarily small number. Consequently, $M_m$ contains a sequence of points such that the maximum of their coordinates with respect to $u_1, \cdots, u_d$ is arbitrary small. 

Let $x$ be an element of $\bbZ^d$ which is congruent to some element of $\calM+\calW$ modulo $\calL$, i.e., $x$ is equal to $m+w+\ell$ for some $m\in \calM, w\in \calW, \ell\in \calL$. Since the set $M_m$ contains a sequence of points such that the maximum of their coordinates with respect to $u_1, \cdots, u_d$ is arbitrary small, it follows that for some large enough positive integers $\lambda_1, \cdots, \lambda_d$, the set $M_m$ contains $m-(\lambda_1 u_1 + \cdots + \lambda_d u_d)$ and $w+\ell+(\lambda_1 u_1 + \cdots + \lambda_d u_d) = w + (\ell + \lambda_1 u_1 + \cdots + \lambda_d u_d)$ belongs to $(\bbN u_1 + \cdots+\bbN u_d)+\calW$. Hence 
$$x = m + w+\ell = (m-(\lambda_1 u_1 + \cdots + \lambda_d u_d)) + (w+\ell+(\lambda_1 u_1 + \cdots + \lambda_d u_d))$$
belongs to 
$$M_m+ ((\bbN u_1 + \cdots+\bbN u_d)+\calW)\subseteq M+W.$$ 
So $M+W$ contains all elements of $\bbZ^d$ which are congruent to some elements of $\calM + \calW$ modulo $\calL$. Moreover, $M+W$ contains $C'$, i.e., it contains all elements of $\bbZ^d$ which are congruent to no element of $\calM+ \calW$ modulo $\calL$. Hence $M+W$ is equal to $\bbZ^d$.

Let $M'$ be a subset of $M$ such that $M' + W = \bbZ^d$. Then each element of $M'+(W\setminus \scrW_1)$ is congruent to some element of $\calM+ \calW$ modulo $\calL$. So $M'+(W\setminus \scrW_1)$ and $C'$ has no common element. Thus $M'+\scrW_1$ contains $C'$. Since $M$ is minimal among the subsets of $C$ with respect to the property that $M+\scrW_1$ contains $C'$, we conclude that $M'$ is equal to $M$. Hence $M$ is a minimal additive complement to $W$ in $\bbZ^d$.
\end{proof}

\begin{remark}
Note that the set $M$ constructed in the proof of Theorem \ref{Thm: Implies existence of min comple} satisfies $M= M_\infty$, i.e., $M_\fin = \emptyset$. 
\end{remark}

\begin{corollary}
\label{Cor:ImageIsSubgrExistsMinComp}
Let $W$ be an eventually periodic subset of $\bbZ^d$ with periods $u_1, \cdots, u_d$. Let $\scrW_1$ be as in Theorem \ref{Thm: existence of min complement implies}. Suppose $\scrW_1$ is nonempty and the map $\pi: \calW \cup \scrW_1 \to \bbZ^d/\calL$ is surjective. Then $W$ has a minimal complement in $\bbZ^d$.
\end{corollary}

\begin{proof}
It follows from Theorem \ref{Thm: Implies existence of min comple} by taking any singleton subset of $\bbZ^d$ as $\calM$.
\end{proof}

\begin{definition}
\label{Defn:Min comp finite}
An ordered pair $(\scrQ_1, \calQ)$ of two nonempty disjoint subsets $\scrQ_1, \calQ$ of $\bbZ^d/\calL$ is said to have a \textnormal{minimal complement} in $\bbZ^d/\calL$ if there exists a nonempty subset $\calN$ of $\bbZ^d/\calL$ such that the following conditions hold. 
\begin{enumerate}
\item The set $\calN + (\calQ \cup \scrQ_1)$ is equal to $\bbZ^d/\calL$.
\item For any $n\in \calN$, there exists $q\in \scrQ_1$ such that $n+q \neq n ' + q'$ for any $n'\in \calN\setminus\{n\}, q'\in \calQ \cup \scrQ_1$. 
\end{enumerate}
\end{definition}

\begin{remark}
Note that if the set $\calN$ as above contains only one element, then the second condition is vacuously true. 
\end{remark}

\begin{lemma}
If $W$ is an eventually periodic subset of $\bbZ^d$ with periods $u_1, \cdots, u_d$ such that the ordered pair $(\pi(\scrW_1), \pi(\calW))$ admits a minimal complement in $\bbZ^d/\calL$, then $W$ has a minimal complement in $\bbZ^d$. 
\end{lemma}

\begin{proof}
It follows from Theorem \ref{Thm: Implies existence of min comple}.
\end{proof}

\begin{corollary}
\label{Cor:Even Perio has min comp}
Suppose any pair of two nonempty disjoint subsets of $\bbZ^d/\calL$ has a minimal complement in $\bbZ^d/\calL$. Then any eventually periodic subset $W$ of $\bbZ^d$ with periods $u_1, \cdots, u_d$ has a minimal complement in $\bbZ^d$ if it is not periodic. 
\end{corollary}

\begin{proof}
It follows from the above Lemma.
\end{proof}

\begin{proposition}
\label{Prop:MinComp of even perio mod 2}
Any eventually periodic subset $W$ of $\bbZ^2$ with periods $(2,0), (0,2)$ has a minimal complement in $\bbZ^2$ if it is not periodic. 
\end{proposition}

To avoid cumbersome notation, the elements of $\bbZ^2/(\bbZ (2,0) + \bbZ (0,2))$ are denoted by $(a,b)$ in the following proof, in stead of denoting them by $\overline{(a,b)}$.

\begin{proof}
By Corollary \ref{Cor:Even Perio has min comp}, it suffices to show that any pair $(\scrQ_1, \calQ)$ of two nonempty disjoint subsets of $\bbZ^2/(\bbZ (2,0) + \bbZ (0,2))$ admits a minimal complement. It is immediate if $\scrQ_1 \cup \calQ$ is equal to $\bbZ^2/(\bbZ (2,0) + \bbZ (0,2))$ ($\calN$ can be taken to be $\{(0,0)\}$, for instance). Note that we may (and do) assume that $\scrQ_1$ contains $\{(0,0)\}$. Thus it remains to consider the case when $\scrQ_1 \cup \calQ$ is not equal to $\bbZ^2/(\bbZ (2,0) + \bbZ (0,2))$ and $\scrQ_1$ contains $\{(0,0)\}$. There are several possibilities for the pair $(\scrQ_1, \calQ)$, which are listed in Table \ref{Table}. For each of them, $\calN$ can be taken as indicated in Table \ref{Table} and for each $n\in \calN$, we could pick the element $q = (0,0)$ of $\scrQ_1$. This completes the proof.
\begin{table}
\begin{tabular}{ccc}
\hline
$\scrQ_1$ & $\calQ$ & $\calN$ \\
\hline \hline
$\{(0,0)\}$ & $\{(1,0)\}$ & $\{(0,0), (0, 1)\}$ \\
$\{(0,0)\}$ & $\{(0,1)\}$ & $\{(0,0), (1, 0)\}$ \\
$\{(0,0)\}$ & $\{(1,1)\}$ & $\{(0,0), (0, 1)\}$ \\
\hline
$\{(0,0)\}$ & $\{(1, 0), (0,1)\} $ & $\{(0,0), (1, 1)\}$ \\
$\{(0,0)\}$ & $\{(1, 0), (1,1)\} $ & $\{(0,0), (0, 1)\}$ \\
$\{(0,0)\}$ & $\{(0, 1), (1,1)\} $ & $\{(0,0), (1, 0)\}$ \\
\hline
$\{(0,0), (1,0)\}$ & $\{(0,1)\} $ & $\{(0,0), (1, 1)\}$ \\
$\{(0,0), (1,0)\}$ & $\{(1,1)\} $ & $\{(0,0), (0, 1)\}$ \\
$\{(0,0), (0,1)\}$ & $\{(1, 0)\} $ & $\{(0,0), (1, 1)\}$ \\
$\{(0,0), (0,1)\}$ & $\{(1, 1)\} $ & $\{(0,0), (1, 0)\}$ \\
$\{(0,0), (1,1)\}$ & $\{(1, 0)\} $ & $\{(0,0), (0, 1)\}$ \\
$\{(0,0), (1,1)\}$ & $\{(0, 1)\} $ & $\{(0,0), (1, 0)\}$ \\
\hline
\end{tabular}
\\[10pt]
\caption{Proposition \ref{Prop:MinComp of even perio mod 2}}
\label{Table}
\end{table}
\end{proof}

The above proposition suggests Proposition \ref{Prop:MinComp Infinite example}, which we prove using the following lemma.

\begin{lemma}
\label{Remark: translation of ordered pair}
Let $(\scrQ_1, \calQ)$ be an ordered pair of two nonempty disjoint subsets of $\bbZ^d/\calL$. Then the following statements hold. 
\begin{enumerate}
\item The pair $(\scrQ_1, \calQ)$ admits a minimal complement if and only if $(\bar v + \scrQ_1, \bar v + \calQ)$ has a minimal complement for any $\bar v \in \bbZ^d/\calL$.
\item If the union $\scrQ_1 \cup \calQ$ is equal to a translate of a subgroup of $\bbZ^d/\calL$, then $(\scrQ_1, \calQ)$ has a minimal complement in $\bbZ^d/\calL$.
\item The pair $(\scrQ_1, \calQ)$ admits a minimal complement if any nontrivial element of the quotient $\bbZ^d/\calL$ has order two (equivalently, $\bbZ^d/\calL$ is isomorphic to a product of copies of $\bbZ/2\bbZ$) and either of the following two conditions hold. 
\begin{enumerate}
\item Each of $\scrQ_1, \calQ$ is a singleton set.
\item The complement of $\scrQ_1 \cup \calQ$ in $\bbZ^d/\calL$ is a singleton set.
\end{enumerate}
\end{enumerate}
\end{lemma}

\begin{proof}
The first statement follows from the definition of minimal complements.

To establish the second statement, it follows from part (1) that it is enough to consider the case when the union $\scrQ_1 \cup \calQ$ is equal to a subgroup $H$ of $\bbZ^d/\calL$. Let $\calN$ denote a set of distinct coset representatives of $H$ in $\bbZ^d/\calL$. Then the conditions (1) and (2) of Definition \ref{Defn:Min comp finite} hold ($q$ can be taken to be any element of $\scrQ_1$). Hence $(\scrQ_1, \calQ)$ has a minimal complement in $\bbZ^d/\calL$.

Suppose any nontrivial element of the quotient $\bbZ^d/\calL$ has order two. By part (1), we may (and do) assume that $\scrQ_1$ contains $(0, \cdots, 0)$. If each of $\scrQ_1, \calQ$ is a singleton set, then $\calQ$ contains a nontrivial element of $\bbZ^d/\calL$ of order two, and consequently, $\scrQ_1\cup \calQ$ is a subgroup of $\bbZ^d/\calL$ and hence $(\scrQ_1, \calQ)$ admits a minimal complement by part (2). Now, assume that the complement of $\scrQ_1 \cup \calQ$ in $\bbZ^d/\calL$ is a singleton set. Let $\bar v$ denote the unique element of $\bbZ^d/\calL$ which does not lie in $\scrQ_1 \cup \calQ$. Note that $\{(0,\cdots, 0), \bar v\} +  (\scrQ_1 \cup \calQ)$ is equal to $\bbZ^d/\calL$. Moreover, for any $q'\in \calQ \cup \scrQ_1$, we have $(0,\cdots, 0) + (0,\cdots, 0) \neq \bar v + q'$ (otherwise, $\bar v$ being an element of order two, would be equal to $q'$, which is a contradiction to the choice of $\bar v$), and $\bar v + (0,\cdots, 0)\neq (0,\cdots, 0) + q'$ (since $\bar v$ does not belong to $\calQ \cup \scrQ_1$). So $(\scrQ_1, \calQ)$ admits a minimal complement.
\end{proof}

\begin{proposition}
\label{Prop:MinComp Infinite example}
Let $W$ be an eventually periodic subset of $\bbZ^d$ with periods $u_1, \cdots, u_d$. Suppose $\scrW_1$ is nonempty. Then $W$ has a minimal complement in $\bbZ^d$ if either of the following conditions hold. 
\begin{enumerate}
\item The image of the map $\pi: \calW \cup \scrW_1 \to \bbZ^d/\calL$ is a translate of a subgroup of $\bbZ^d/\calL$.
\item Any nontrivial element of the quotient $\bbZ^d/\calL$ has order two\footnote{Equivalently, $\bbZ^d/\calL$ is isomorphic to a product of copies of $\bbZ/2\bbZ$.} and any one of the following conditions hold. 
\begin{enumerate}
\item Each of $\pi(\scrW_1), \pi(\calW)$ is a singleton set.
\item The complement of $\pi(\scrW_1 \cup \calW)$ in $\bbZ^d/\calL$ is a singleton set.
\end{enumerate}
\end{enumerate}
\end{proposition}

\begin{proof}
It follows from Theorem \ref{Thm: Implies existence of min comple} and Lemma \ref{Remark: translation of ordered pair}.
\end{proof}

Using the above Proposition, we write down infinitely many examples of infinite subsets of $\bbZ^d$ admitting minimal complements. 

\begin{example}
\label{Example infinite}
Let $e_i$ denote the vector $(0, \cdots, 1, \cdots, 0)$ in $\bbZ^d$ whose $i$-th coordinate is equal to one and the remaining coordinates are equal to zero. 
\begin{enumerate}
\item For any positive integer $k\geqslant 1$ and $1\leqslant i\leqslant d$, the eventually periodic set 
$$
\{(0,\cdots, 0)\} \sqcup (\{ e_i, 2e_i, \cdots, (k-1)e_i\} + (\bbN (k e_1) + \cdots + \bbN (k e_d)))$$
has a minimal complement in $\bbZ^d$ by Proposition \ref{Prop:MinComp Infinite example}(1). 

\item For any $1\leqslant i\leqslant d$, the eventually periodic set 
$$
\{(0,\cdots, 0)\} \sqcup (\{ e_i\} + (\bbN (2 e_1) + \cdots + \bbN (2 e_d)))$$
has a minimal complement in $\bbZ^d$ by Proposition \ref{Prop:MinComp Infinite example}(2(a)). 

\item 
Let $d\geqslant 2$ and $F$ be a subset of $\bbZ^d$ containing $2^d-2$ elements such that its image under $\bbZ^d \to \bbZ^d/(\bbZ (2 e_1) + \cdots + \bbZ (2 e_d))$ also contains $2^d-2$ elements. Then 
$$
\{(0,\cdots, 0)\} \sqcup (F+ (\bbN (2 e_1) + \cdots + \bbN (2 e_d)))$$
has a minimal complement in $\bbZ^d$ by Proposition \ref{Prop:MinComp Infinite example}(2(b)). 
\end{enumerate}

\end{example}

\subsection{Equivalent condition}

\begin{theorem}
\label{Thm: Even peri 1 has min comp iff}
Let $W$ be an eventually periodic subset of $\bbZ^d$ with periods $u_1, \cdots, u_d$. Let $\scrW_1$ be as in Theorem \ref{Thm: existence of min complement implies}. Suppose $\scrW_1$ contains only one element. Then $W$ has a minimal complement in $\bbZ^d$ if and only if there exists nonempty finite subset $\calM$ of $\bbZ^d$ satisfying the following. 
\begin{enumerate}
\item The map $\pi: \calM\to \bbZ^d/\calL$ is injective.
\item The map $\pi: (\calM+ \calW\cup \scrW_1)) \to \bbZ^d/\calL$ is surjective.
\item For any $m\in \calM$, there exists $w\in \scrW_1$ such that $m + w \not\equiv m' + w' \modu \calL$ for any $m'\in \calM\setminus\{m\}, w'\in \calW\cup \scrW_1$.
\end{enumerate}
\end{theorem}

\begin{proof}
If $\bbZ^d$ has a subset $\calM$ such that the conditions (1), (2), (3) hold, then by Theorem \ref{Thm: Implies existence of min comple}, $W$ has a minimal complement in $\bbZ^d$. 

Suppose $W$ has a minimal complement $M$ in $\bbZ^d$. Let $\calM$ be as in Theorem \ref{Thm: existence of min complement implies}. Then $\calM$ is a nonempty finite set. Since the composite map $\calM \hra M_\infty \twoheadrightarrow \pi(M_\infty)$ is bijective, it follows that the map $\pi: \calM\to \bbZ^d/\calL$ is injective. By Theorem \ref{Thm: existence of min complement implies}, condition (2) follows, and for any $m\in \calM$, there exists $w\in \scrW_1$ such that $m + w \not\equiv m' + w' \modu \calL$ for any $m'\in \calM, w'\in \calW$, which gives $m + w \not\equiv m' + w' \modu \calL$ for any $m'\in \calM\setminus\{m\}, w'\in \calW$. Note that for $m'\in \calM\setminus\{m\}, w'\in \scrW_1$, we obtain $m+w  \not\equiv m' + w' \modu \calL$ (since $\scrW_1$ contains only one element). This proves condition (3).
\end{proof}

\section{Minimal complements in finitely generated abelian groups}
\label{Sec: Minimal complement fngen}
In this section $G$ is a finitely generated abelian group. We have already seen in Theorem \ref{theorem1.1} that for finite sets $W\subset G$, a minimal complement of $W$ always exists. So we shall consider infinite subsets $W\subset G$ throughout this section. By the structure theorem for finitely generated abelian groups, 
$$G\simeq \mathbb{Z}^{d}\times(\mathbb{Z}/a_{1}\mathbb{Z})\times \cdots \times(\mathbb{Z}/a_{s}\mathbb{Z})$$
with $d\geqslant 0$ and where $a_{1}|a_{2}|...|a_{s}$ are positive integers $>1$, determined uniquely by the isomorphism type of the group $G$. $\mathbb{Z}^{d}$ is said to be the torsion free part and $F :=(\mathbb{Z}/a_{1}\mathbb{Z})\times \cdots \times(\mathbb{Z}/a_{s}\mathbb{Z})$ is the torsion part. $G$ is finite if and only if $d=0$. Since we are in the realm of infinite subsets $W\subset G$, so $G$ is infinite. Hence, $d>0$ in this section. We first show the following proposition which describes the behaviour of minimal complements of product sets.

\begin{proposition}\label{MinCompProd}
Let $n$ be a positive integer. Let $G_1, \cdots, G_n$ be groups. For each $1\leqslant i\leqslant n$, let $A_i$ be a subset of $G_i$ with minimal complement $M_i$. Then $\prod_{1\leqslant i\leqslant n}M_{i}$ is a minimal complement of $\prod_{1\leqslant i\leqslant n}A_{i}$ in $\prod_{1\leqslant i\leqslant n}G_{i}$.
\end{proposition}
\begin{proof}
Let $A_{1}$ and $A_{2}$ be two given sets in the groups $G_{1}$ and $G_{2}$ with minimal complements $M_{1}$ and $M_{2}$ respectively. Then
\begin{align*}
A_{1}.M_{1} & = G_{1}, \,\, A_{1}.\big\lbrace M_{1}\setminus \lbrace m\rbrace \big\rbrace \subsetneq G_{1} \,\,\forall m\in M_{1} ,\\
A_{2}.M_{2} & = G_{2}, \,\, A_{2}.\big\lbrace M_{2}\setminus \lbrace m\rbrace \big\rbrace \subsetneq G_{2} \,\,\forall m\in M_{2} .
\end{align*}
Now $A_{1}\times A_{2}\subset G_{1}\times G_{2}$. First we show that $M_{1}\times M_{2}$ is a complement of $A_{1}\times A_{2}$ in $G_{1}\times G_{2}$. Pick any element $(g_{1},g_{2})\in G_{1}\times G_{2}$. Then $\exists a_{1}\in A_{1},a_{2}\in A_{2}, m_{1}\in M_{1},m_{2}\in M_{2}$ with $a_{1}m_{1} = g_{1}$ and $a_{2}m_{2} = g_{2}$. Thus any $(g_{1},g_{2})\in G_{1}\times G_{2}$ can be represented as $(a_{1}m_{1},a_{2}m_{2})$ for some $a_{1}\in A_{1},a_{2}\in A_{2}, m_{1}\in M_{1},m_{2}\in M_{2}$, i.e.,
$$(A_{1}\times A_{2}). (M_{1}\times M_{2}) = G_{1}\times G_{2},$$
thus it is a complement of $A_{1}\times A_{2}$. 

Now, we show that $M_{1}\times M_{2}$ is minimal. Remove an element $(m,n)$ from $M_{1}\times M_{2}$ and look at the set $ M :=M_{1}\times M_{2} \setminus \lbrace (m,n)\rbrace $. We show that $M$ is not a complement of $A_{1}\times A_{2}$ in $G_{1}\times G_{2}$, i.e., $(A_{1}\times A_{2}).M\subsetneq G_{1}\times G_{2}$. Since $M_{1}$ is a minimal complement of $A_{1}$, $\exists a_{1}\in A_{1},g_{1}\in G_{1}$ such that the only way of representing $g_{1}$ in $A_{1}.M_{1}$ is $a_{1}m$. Similarly, $\exists a_{2}\in A_{2},g_{2} \in G_{2}$ with $g_{2} = a_{2}n$. We show that this $(g_{1},g_{2}) \notin (A_{1}\times A_{2}).M$. Indeed, this is clear from the fact that $(g_{1},g_{2}) $ can only be represented in $(A_{1}\times A_{2}). (M_{1}\times M_{2})$ as $(a_{1}m,a_{2}n)$.

To prove the general case we use induction.	

Without loss of generality, let us assume that the statement is true for $k$ groups $G_{1},G_{2},\cdots, G_{k}$ with $k<n$. We show that the statement holds for $(k+1)$-groups. By hypothesis,
\begin{align*}
A_{1} . M_{1}  = G_{1}, \,\, A_{1} . M_{1}\setminus & \lbrace m_{1}\rbrace  \subsetneq G_{1} \,\,\forall m_{1}\in M_{1} \\
A_{2} . M_{2}  = G_{2}, \,\, A_{2} . M_{2}\setminus & \lbrace m_{2}\rbrace  \subsetneq G_{2} \,\,\forall m_{2}\in M_{2} \\	
\vdots	\\
A_{k} . M_{k}  = G_{k}, \,\, A_{k} . M_{k}\setminus & \lbrace m_{k}\rbrace \subsetneq G_{k} \,\,\forall m_{k}\in M_{k}\\
A_{k+1} . M_{k+1}  = G_{k+1}, \,\, A_{k+1} . M_{k+1}\setminus & \lbrace m_{k+1}\rbrace \subsetneq G_{k+1} \,\,\forall m_{k+1}\in M_{k+1}.
\end{align*}
By the inductive assumption $N :=M_{1}\times \cdots \times M_{k}$ is a minimal complement of $A_{1}\times \cdots \times A_{k}$ in $G_{1}\times \cdots \times G_{k}$. To show that $N\times M_{k+1}$ is a minimal complement of  $A_{1}\times \cdots \times A_{k+1}$ in $G_{1}\times \cdots \times G_{k+1}$. It is clear that $N\times M_{k+1}$ is a complement. To show that it is minimal, we remove an arbitrary point $(x,y)$ from $N\times M_{k+1}$ and argue as above to get the required statement.
\end{proof}

An immediate consequence of the above proposition (combined with previous work of Chen--Yang \cite{ChenYang12} and Kiss--S\'andor--Yang \cite{KissSandorYangJCT19}) is the construction of infinite sets in $\mathbb{Z}^{d}$ having minimal complements. We first state their results - 
\begin{theorem}[Chen--Yang {\cite[Theorems 1, 2]{ChenYang12}}]\label{Chen-Yang}
Let $W$ be a nonempty subset of $\bbZ$. Suppose either of the following holds.
\begin{enumerate}
\item $\inf W = -\infty, \,\,\sup W = +\infty$.
\item $W = \{1 = w_1 < w_2 <\cdots \}$ and $\limsup_{i\rightarrow \infty} (w_{i+1}-w_{i}) = +\infty$.
\end{enumerate}
Then $W$ has a minimal complement in $\mathbb{Z}$.
\end{theorem}

\begin{theorem}[Kiss--S\'andor--Yang {\cite[Theorem 3]{KissSandorYangJCT19}}]\label{Kiss-Sandor-Yang}
Let $W\subset \mathbb{Z}$ be of the following form, $$W = (m\mathbb{N} + X_{m}) \cup Y^{(0)} \cup Y^{(1)},$$
where $X_{m}\subseteq \lbrace 0, 1,\cdots, m-1\rbrace, Y^{(0)}\subseteq \mathbb{Z}^{-}, Y^{(1)} $ are finite sets with $Y^{(0)} \mod m \subseteq X_{m}$ and $(Y^{(1)}\mod m) \cap X_{m} = \emptyset .$ There exists a minimal complement to $W$ if and only if there exists $T\in \mathbb{Z}^{+}, m |T$, and $C\subseteq \lbrace 0, 1, \cdots, T-1\rbrace $ such that
\begin{enumerate}
\item 	$C+(X_{T}\cup Y^{(1)} ) \mod T = \lbrace 0, 1,\cdots, T-1\rbrace$ where $X_{T}=\cup_{i=0}^{\frac{T}{m}-1}\lbrace im +X_m\rbrace ;$
\item for any $c\in C$, there exists $y\in Y^{(1)} $ for which $c +y	\not\equiv c'+x \, (\mathrm{mod} \, T),$ where $c'\in C\setminus{c} $ and $x \in X_{T}$.
\end{enumerate}
\end{theorem}

Using Theorem \ref{Chen-Yang}, Theorem \ref{Kiss-Sandor-Yang} and Proposition \ref{MinCompProd}, we have the following immediate result.

\begin{theorem}[Product sets having minimal complements in $\mathbb{Z}^{d}$]\label{MinProdThm}
Fix $d$ a positive integer. Let $W_{i}\subset \mathbb{Z}$ be of the form described in Theorem \ref{Chen-Yang} or Theorem \ref{Kiss-Sandor-Yang} or finite, for each $1\leqslant i\leqslant d$. Then $W_{1}\times W_{2}\times \cdots \times W_{d}$ has a minimal complement in $\mathbb{Z}^{d}$. 
\end{theorem}

\begin{proof}
Each $W_{i}\subset \mathbb{Z}$ has a minimal complement $M_{i}$ in $\mathbb{Z}$ (using Theorem \ref{Chen-Yang} or Theorem \ref{Kiss-Sandor-Yang} or Theorem \ref{theorem1.1}). Using Proposition \ref{MinCompProd}, we see that $M_{1}\times M_{2}\times \cdots \times M_{d}$ is a required minimal complement.
\end{proof}

However, not all infinite sets in $\mathbb{Z}^{d}$ are product sets. We have seen in the previous section a sufficient condition for the existence of minimal complements for eventually periodic sets in $\mathbb{Z}^{d}$ (which are not product sets, when they are not periodic). We shall now exploit the structure of the finitely generated abelian group $G$.

\begin{lemma}\label{lemmasec5}
Let $G\simeq \mathbb{Z}^{d}\times F$ be any finitely generated abelian group ($F$ is the finite torsion part). Let $\emptyset\neq A\subset \mathbb{Z}^{d}$. Suppose minimal complement of $A$ in $\mathbb{Z}^{d}$ exist. Then a minimal complement of $A\times H $ exists $\forall \,\emptyset\neq H\subseteq F$. 
\end{lemma}

\begin{proof}
We have $G\simeq \mathbb{Z}^{d}\times F$, with $F$ a finite group. Suppose $\emptyset\neq A\subset \mathbb{Z}^{d}, \emptyset\neq H\subset F$.

Let $B$ be a minimal complement of $A$ in $\mathbb{Z}^{d}$. Since $F$ is finite, $H$ is also finite. By Theorem \ref{theorem1.1}, we know that a minimal complement of $H$ exists. Let it be $H'$. Now using the previous proposition, we have that $B\times H'$ is a minimal complement of $A\times H$ in $G$.	
\end{proof}

We come to the main theorem which describes a large class of infinite sets having minimal complements in finitely generated abelian groups.
\begin{theorem}[Minimal complements in finitely generated abelian groups]\label{sec5Thm}
Let $G\simeq \mathbb{Z}^{d}\times F$ be any finitely generated abelian group. Suppose $W\subseteq \mathbb{Z}^{d}$ be either of the form given in Theorem \ref{Thm: Implies existence of min comple} or a product set $W_{1}\times W_{2}\times \cdots \times W_{d}$ as described in Theorem \ref{MinProdThm}. Then $W\times H$ will have a minimal complement in $G$ where $H \subseteq F$ is any arbitrary nonempty subset. 
\end{theorem}

\begin{proof}
The form of $W\subseteq \mathbb{Z}^{d}$ ensures that $W$ has a minimal complement in $\mathbb{Z}^{d}$. After this we apply Lemma \ref{lemmasec5} to get the desired conclusion.
\end{proof}

\begin{remark}
\label{Remark:MinCompEg}
The upshot of the discussion in secton \ref{Sec: Minimal complement fngen} is to provide examples of subsets of finitely generated abelian groups admitting minimal complements. Some such examples were already given in secton \ref{Sec: Minimal complement} (see Theorem \ref{Thm: Implies existence of min comple}), where we considered the group $\bbZ^d$, i.e., free abelian groups only. The immediate question is look for examples of such subsets in finitely generated abelian groups having nontrivial torsion. By Proposition \ref{MinCompProd}, a subset $W$ of a group $G$ admits a minimal complement if $G$ is isomorphic to the direct product of groups $G_1\times \cdots \times G_n$ and under such an isomorphism $W$ corresponds to the product of subsets $W_i$ of $G_i$ having minimal complements. 
Hence, as a consequence of Theorems \ref{Chen-Yang}, \ref{Kiss-Sandor-Yang}, \ref{Thm: Implies existence of min comple}, \ref{theorem1.1} and Propositions  \ref{prop4.1}, \ref{prop6.1}, it follows that a subset $W$ of a finitely generated abelian group $G$ admits a minimal complement in $G$ if there exists finitely generated free abelian groups $G_1, \cdots, G_n$ and an isomorphism $\psi: G\xra{\sim} G_\tors \times G_1\times \cdots \times G_n$ (where $G_\tors$ denotes the torsion part of $G$) such that $\psi(W)$ is equal to the product $W_0 \times W_1 \times \cdots \times W_n$ where $W_0$ is a nonempty subset of $G_\tors$, $W_1, \cdots, W_n$ are subsets of $G_1, \cdots, G_n$ and for each $1\leqslant i\leqslant n$, one of the following conditions hold.
\begin{enumerate}
\item $G_i=\bbZ$ and $W_i$ satisfies the conditions of Theorem \ref{Chen-Yang}. 
\item $G_i = \bbZ$ and $W_i$ satisfies the conditions of Theorem \ref{Kiss-Sandor-Yang}. 
\item $G_i = \bbZ^{d_i}$ for some integer $d_i\geqslant 1$ and $W_i$ satisfies the conditions of Theorem \ref{Thm: Implies existence of min comple}.
\item $G_i = \bbZ^{d_i}$ for some integer $d_i\geqslant 1$ and $W_i$ is a nonempty finite subset of $G_i$. 
\item $G_i$ is a free abelian group of finite rank and $W_i$ is a subgroup of $G_i$.
\item $G_i = \bbZ^{d_i}$ for some integer $d_i\geqslant 1$ and $W_i$ is equal to $\{(\pm n,\cdots, \pm n)\in \bbZ^{d_i}\,|\, n\in \bbZ\}$ (see Proposition \ref{prop6.1}). 
\end{enumerate}
\end{remark}

To conclude this section, we see that a combination of Theorem \ref{sec5Thm}, Theorem \ref{MinProdThm} and Proposition \ref{MinCompProd} gives us infinite sets in arbitrary finitely generated abelian groups having minimal complements, providing a partial answer to Nathanson's Question \ref{nathansonprob13}.

\section{Conclusion and further remarks}
\label{Sec:Conclusion}

Finally, we conclude the article with a class of examples of infinite sets which are not eventually periodic (i.e., they don't fall in the class of  Theorem \ref{Thm: existence of min complement implies} and Theorem \ref{Thm: Implies existence of min comple}) but have minimal complements. An example of such a subset $W$ of $\bbZ$ is given in \cite[Theorem 4]{KissSandorYangJCT19}. By Proposition \ref{MinCompProd}, its $d$-fold product $W\times \cdots \times W$ has a minimal complement in $\bbZ^d$ and it is not eventually periodic. 
In Propositions \ref{prop6.1}, \ref{prop6.2}, we give examples of non-eventually periodic subsets of $\bbZ^d$, which are algebro-geometric in nature and admit minimal complements. These are obtained by taking images of $\bbZ^m$ under polynomial maps. 
\begin{proposition}\label{prop6.1}
Let $d\geqslant 2$ be a positive integer. Then the subset 
$$\calD := \{(\pm n,\cdots, \pm n)\in \bbZ^d\,|\, n\in \bbZ\}$$
of $\bbZ^d$ is not eventually periodic. For each $1\leqslant i\leqslant d$, the hypersurface defined by $x_i=0$, i.e., the set 
$$\calH_i:=\{(x_1, \cdots, x_d)\in \bbZ^d\,|\, x_i=0\}$$ is a minimal complement of $\calD$.
\end{proposition}

\begin{proof}
Suppose $\calD$ is eventually periodic. So there exist elements $u_1, \cdots, u_d$ in $\bbZ^d$ satisfying no nontrivial $\bbZ$-linear relation and an integer $n_0\geqslant 1$ such that for any integer $n$ with $|n|> n_0$, 
$$\calD \supseteq (\pm n, \cdots, \pm n) + (\bbN u_1 + \cdots + \bbN u_d).$$ 
Let $n'$ denote the smallest integer $\geqslant \max \{n_0 , ||u_1||, \cdots, ||u_d||\}$. For any positive integer $n$ with $n>\sqrt d n'$ and any $1\leqslant i\leqslant d$, the vector $(n, \cdots, n)+u_i$ is contained in the open ball $B_{(n, \cdots, n), n/\sqrt d}$ in $\bbR^d$ with center at the point $(n,\cdots, n)$ with radius $n/\sqrt d$, and it is also contained in $\calD$. Hence the vectors $u_1, \cdots, u_d$ are scalar multiples of $(1, \cdots, 1)$, which is absurd. So $\calD$ is not eventually periodic.

For any element $(x_1, \cdots, x_d)$ of $\bbZ^d$, the point $(x_i, \cdots, x_i)$ belongs to $\calD$ and $\calH_i$ contains $(x_1, \cdots, x_d) - (x_i, \cdots, x_i)$ and hence $\calH_i$ is a complement of $\calD$. 
Suppose $M_i$ is a subset of $\calH_i$ such that $M_i + \calD = \bbZ^d$. Note that the $i$-th coordinate of any element of $M_i + \calD \setminus\{(0,\cdots,0)\}$ is nonzero. So it contains no element of $\calH_i$. Thus $\calH_i$ is contained in $(0,\cdots, 0) + M_i$. So $M_i$ is equal to $\calH_i$. Hence $\calH_i$ is a minimal complement to $\calD$. 
\end{proof}

\begin{proposition}
\label{prop6.2}
Let $d\geqslant 2$ be a positive integer. Let $f_1, \cdots, f_d$ be elements of $\bbZ[X_1, \cdots, X_m]$. 
Define the subset $\calS$ of $\bbZ^d$ by 
$$\calS := \{( f_1(n_1, \cdots, n_m),\cdots, f_d(n_1, \cdots, n_m))\,|\, (n_1, \cdots, n_m)\in \bbZ^m\}.$$
Let $1\leqslant i\leqslant d$ be an integer such that the map $f_i:\bbZ^m\to \bbZ$ is surjective. Then the hypersurface defined by $x_i=0$, i.e., the set 
$$\calH_i:=\{(x_1, \cdots, x_d)\in \bbZ^d\,|\, x_i=0\}$$
is a complement of $\calS$. Moreover, if $\calS \cap \calH_i$ contains only one element, then $\calH_i$ is a minimal complement of $\calS$. 
\end{proposition}

\begin{proof}
For each element $(x_1, \cdots, x_d)$ of $\bbZ^d$, there exists an element $(n_1, \cdots, n_m)\in \bbZ^m$ such that $f_i(n_1, \cdots, n_m) = x_i$. So $\calS$ contains the point $(f_1(n_1, \cdots, n_m), \cdots,$ $ f_d(n_1, \cdots, n_m))$ and $\calH_i$ contains the difference $(x_1, \cdots, x_d)-(f_1(n_1, \cdots, n_m), \cdots,$ $ f_d(n_1, \cdots, n_m))$. Hence $\calH_i$ is a complement to $\calS$. 

Suppose $\calS \cap \calH_i$ contains only one element $P$ of $\bbZ^d$. Let $M_i$ be a subset of $\calH_i$ such that $M_i + \calS = \bbZ^d$. Note that the $i$-th coordinate of any point of $M_i + \calS \setminus\{P\}$ is nonzero. So $\calH_i$ is contained in $P + M_i$. Hence $\calH_i - P = \calH_i$ is contained in $M_i$. Hence $\calH_i$ is a minimal complement of $\calS$. 
\end{proof}

\section{Acknowledgements}
The first author would like to thank the Fakult\"at f\"ur Mathematik, Universit\"at Wien where a part of the work was carried out.
The second author would like to acknowledge the Initiation Grant from the Indian Institute of Science Education and Research Bhopal, and the INSPIRE Faculty Award from the Department of Science and Technology, Government of India.

\def\cprime{$'$} \def\Dbar{\leavevmode\lower.6ex\hbox to 0pt{\hskip-.23ex
  \accent"16\hss}D} \def\cfac#1{\ifmmode\setbox7\hbox{$\accent"5E#1$}\else
  \setbox7\hbox{\accent"5E#1}\penalty 10000\relax\fi\raise 1\ht7
  \hbox{\lower1.15ex\hbox to 1\wd7{\hss\accent"13\hss}}\penalty 10000
  \hskip-1\wd7\penalty 10000\box7}
  \def\cftil#1{\ifmmode\setbox7\hbox{$\accent"5E#1$}\else
  \setbox7\hbox{\accent"5E#1}\penalty 10000\relax\fi\raise 1\ht7
  \hbox{\lower1.15ex\hbox to 1\wd7{\hss\accent"7E\hss}}\penalty 10000
  \hskip-1\wd7\penalty 10000\box7}
  \def\polhk#1{\setbox0=\hbox{#1}{\ooalign{\hidewidth
  \lower1.5ex\hbox{`}\hidewidth\crcr\unhbox0}}}
\providecommand{\bysame}{\leavevmode\hbox to3em{\hrulefill}\thinspace}
\providecommand{\MR}{\relax\ifhmode\unskip\space\fi MR }
\providecommand{\MRhref}[2]{%
  \href{http://www.ams.org/mathscinet-getitem?mr=#1}{#2}
}
\providecommand{\href}[2]{#2}

\end{document}